\documentclass[12pt]{amsart}

\usepackage[a4paper,left=3.2cm,right=2.3cm,top=3.5cm,bottom=3.4cm]{geometry}

\usepackage{caption}
\usepackage{longtable}
\usepackage{graphicx}
\usepackage{epsfig}
\usepackage{color}
\usepackage{array}
\usepackage{amsmath}
\usepackage{amsthm}
\usepackage{amssymb}
\usepackage{enumerate}
\usepackage{amscd}
\usepackage[bookmarksnumbered,bookmarksopen,pdfusetitle,unicode]{hyperref}
\usepackage[all,cmtip]{xy}
\usepackage{pgf,tikz}
\usepackage{epsfig}
\usepackage{newlfont}
\usepackage[english]{babel}

\usepackage{mathrsfs}
\usepackage{color}
\usepackage{xcolor}
\usetikzlibrary{arrows}
\usepackage{float}
\usepackage{stix}
\usetikzlibrary {positioning}

\numberwithin{equation}{section}

\theoremstyle{plain}
\newtheorem{thm}{Theorem}[section]
\newtheorem{cor}[thm]{Corollary}
\newtheorem{lem}[thm]{Lemma}
\newtheorem{prop}[thm]{Proposition}

\theoremstyle{definition}
\newtheorem{rem}[thm]{Remark}
\newtheorem{ex}[thm]{Example}
\newtheorem{defn}[thm]{Definition}

\renewcommand {\epsilon}{\varepsilon}
\renewcommand {\le}{\leqslant}
\renewcommand {\ge}{\geqslant}
\renewcommand {\leq}{\leqslant}
\renewcommand {\geq}{\geqslant}

\graphicspath{ {./figures/} }
\usepackage{dynkin-diagrams}
\usepackage[all]{xy}
\usepackage{color}

\usepackage{subfig}

\usepackage{epstopdf}
\begin{document}

\title{Matrix factorization of quotient singularities of type A and D} 

\author{Ayse Sharland}
\address{Department of Mathematics and Applied Mathematical Sciences, University of Rhode Island, Kingston, RI 02881, US}
\email{aysharland@uri.edu}

\author{Meral Tosun}
\address{Department of Mathematics, Galatasaray University, Ortak{\"o}y 34357, Istanbul, T\"urkiye \\
Kavli Institute for Physics and Mathematics of the Universe (WPI), The University of Tokyo, 5-1-5 Kashiwanoha, Kashiwa, Chiba, 277-8583, Japan}
\email{mrltosun@gmail.com}

\subjclass[2000]{Primary: 14B05, Secondary: 13C14, 13A50}


\begin{abstract} In this article, we extend the classical relationship between ADE singularities and their maximal Cohen-Macaulay modules to quotient singularities of types A and D. We identify projected hypersurfaces which allow us to construct explicit matrix factorizations, and in return, families of maximal Cohen-Macaulay modules. These hypersurfaces also yield the dual resolution graphs of the corresponding quotient singularities via Oka's method. Finally, we show that all special Cohen-Macaulay modules over the original quotient singularity can be recovered from those over the projected hypersurfaces. \end{abstract}

\maketitle

\begin{center}
\emph{To the memory of L{\^e} D{\~u}ng Tr{\'a}ng, with  gratitude.}
\end{center}

\tableofcontents

\parskip12pt

\section{Introduction}

Quotient singularities occupy a central position in algebraic geometry, singularity theory, and representation theory. They arise as quotients of smooth varieties by finite group actions and constitute one of the most fundamental classes of rational singularities. Besides their intrinsic geometric interest, they appear naturally in the minimal model program, the McKay correspondence, moduli theory, and mathematical physics. 

The local classification of quotient singularities arising from discontinuous group actions was given by Prill (\cite{Prill}). Brieskorn's classification \cite{Brieskorn} of quotient surface singularities associated with finite small subgroups of $\mathrm{GL}(2,\mathbb C)$ organizes them into cyclic, dihedral, and exceptional families. In the Gorenstein case, corresponding to finite subgroups of $\mathrm{SL}(2,\mathbb C)$, one obtains the classical ADE (or Kleinian) surface singularities. These provide the first and best understood examples of quotient singularities. Together with Artin's work on the geometry of minimal resolutions (Dynkin diagrams), they reveal deep connections between geometry, Lie theory, and representation theory. 

In the following years, the study of ADE singularities advanced further by the classical McKay correspondence due to McKay (\cite{McKay}) who observed a remarkable relation between finite subgroups of $\textnormal{SL}(2,\mathbb C)$ and the Dynkin diagrams (also see \cite{GonzalezSprinbergVerdier} for a geometric realization of this correspondence). In this article, we focus on maximal Cohen-Macaulay modules (MCM modules, for short) in parallel with matrix factorizations. MCM modules are particularly important since their depth is equal to the dimension of the ring and they provide information closely related to the resolution of the singularity. 

If a singularity is defined by a hypersurface equation, MCM modules can explicitly be described by matrix factorizations. In \cite{Eisenbud}, Eisenbud showed that matrix factorizations of the defining equation give rise to MCM modules over the hypersurface ring. This connection has since become an important tool in the study of hypersurface singularities and their CM modules (see, for example, \cite{Knorrer,BuchweitzGreuelSchreyer,Yoshino}). For the classical ADE singularities, the indecomposable MCM modules can be described explicitly by matrix factorizations as studied in \cite{KajiuraSaito}. In general, however, quotient singularities are not hypersurfaces. Nevertheless, these results motivate our focus on projected hypersurface equations. While projected hypersurfaces are non-isolated, their Cohen-Macaulay (CM, for short) representation theory can be considerably richer (see, for example, \cite{BurbanDrozd}). Therefore, it is natural to ask whether one can still use matrix factorizations on projected hypersurfaces to recover the \textit{special} CM modules (Definition \ref{def-spec}) of the original quotient singularity. The aim of this paper is to answer this question for cyclic and dihedral quotient singularities.

After recalling some preliminary definitions and results in Section \ref{sect-prem}, we cover the construction for cyclic quotient singularities in two categories: the Gorenstein and non-Gorenstein case. For the former, as studied in \cite{KajiuraSaito}, the MCM modules can be described directly in terms of matrix factorizations. We recall these results in Section \ref{sect-gor}, and then study the non-Gorenstein case in detail. 

For quotient singularities which are not necessarily Gorenstein, not every indecomposable MCM module directly corresponds to an exceptional curve. Wunram (\cite{Wunram}) introduced the class of special CM modules and proved that the non-free indecomposable special CM modules are in bijection with the irreducible exceptional curves of the minimal resolution. Thus,  special CM modules form a distinguished finite collection inside the category of MCM modules and reflect the geometry of the exceptional divisor. This point of view was further developed by Iyama and Wemyss \cite{IyamaWemyss} and, in the cyclic and dihedral cases, through the reconstruction algebras of type A and D in \cite{WemyssA, WemyssDI, WemyssDII}. More recently, an almost complete classification of indecomposable reflexive modules over quotient surface singularities was obtained in \cite{ArciniegaCisnerosRomano}.

Furthermore, in the case of non-Gorenstein cyclic and dihedral singularities, the invariant rings are no longer hypersurfaces. For those types, we start with finding projected hypersurfaces to build matrix factorizations (see equation \ref{eq-anq} and Theorem \ref{projected-hypersurface-D}). We also choose those equations so that one can recover the dual resolution graphs of the corresponding singularities using Oka's combinatorial method introduced in \cite{oka1} (also see \cite{oka2} and \cite{MAG})\footnote{For the cyclic case, we need to consider a linear transformation and blowing down $(-1)$-curves to match the dual resolution graphs. We note these in Remark \ref{rem-anq-oka}.}

For non-Gorenstein cyclic singularities, we construct an explicit finite family of $2\times 2$ matrix factorizations in Proposition \ref{cyclic-mf} and identify the reflexive hulls of the resulting rank-one MCM modules over the original quotient ring in Theorem \ref{cyclic-comparison}. This way, all indecomposable special MCM modules of the cyclic quotient singularity are recovered. We note that the projected hypersurface may have further indecomposable MCM modules, but these are not needed for this purpose. Moreover, the choice of the finite family is not unique: as shown in Example \ref{ex1}, non-isomorphic MCM modules over the projected hypersurface may give rise to the same reflexive module over the original quotient singularity.

We next study dihedral quotient singularities in Sections \ref{sect-dnq}-\ref{sect-dnq2}. In order to find a projected hypersurface, we use the explicit generators and relations given by Riemenschneider in \cite{riemen} (also see Proposition \ref{reduced-ideal})  and obtain a projected hypersurfaces associated with the original quotient singularity (Theorem \ref{projected-hypersurface-D}). Then we construct families of $2\times 2$ and $4\times 4$ matrix factorizations of the projected hypersurface, giving rank-one and rank-two MCM modules, respectively. Using the descriptions of special CM modules in \cite{WemyssDI,WemyssDII}, we show how their reflexive hulls recover the special CM modules associated with the exceptional curves of the minimal resolution. Thus, although passing to a projected hypersurface introduces additional MCM modules, a finite collection of explicit matrix factorizations is sufficient to recover the special CM modules of the original quotient singularity.


\section{Preliminaries}\label{sect-prem}

\subsection{Matrix factorizations} Let $S=\mathbb C[x_1,\ldots,x_n]$ and let $f\in S$ be a non-zero polynomial.
Matrix factorizations were introduced by Eisenbud \cite{Eisenbud} in his
seminal work on free resolutions over hypersurface rings. He proved that if
$R:=S/(f)$ is a hypersurface ring, then every MCM
$R$-module admits an eventually $2$-periodic minimal free resolution. The
periodic part is determined by a pair of matrices $(\varphi,\psi)$ satisfying
\[\psi\varphi=\varphi\psi=f\,\operatorname{Id}\] which establishes a fundamental connection between hypersurface singularities and MCM modules.

A matrix factorization of $f$ consists of two finite rank free $S$-modules $M_0,M_1$ together with homomorphisms
$$
\varphi\colon M_0\longrightarrow M_1,\qquad \psi\colon M_1\longrightarrow M_0$$
such that $\psi\varphi=f\,\mathrm{Id}_{M_0}$ and $\varphi\psi=f\,\mathrm{Id}_{M_1}$. 
Equivalently, it is a $\mathbb Z/2\mathbb Z$-graded free module
$M=M_0\oplus M_1$ equipped with the odd differential
$$
d=\begin{pmatrix}
0&\psi\\
\varphi&0
\end{pmatrix},
\qquad
d^2=f\,\mathrm{Id}_M
$$
We denote the category of matrix factorizations of $f$ by $\mathrm{MF}(S,f)$. Throughout this paper we consider only \textit{reduced} matrix factorizations, those having no trivial direct summands. A morphism
$$
(m_0,m_1)\colon (\varphi_M,\psi_M)\longrightarrow (\varphi_N,\psi_N)$$
consists of homomorphisms
$m_0\colon M_0\to N_0$ and $m_1\colon M_1\to N_1$ such that
$$
m_1\varphi_M=\varphi_Nm_0, \qquad m_0\psi_M=\psi_Nm_1$$
The category $\mathrm{MF}(S,f)$ is additive and carries the shift functor
$$(M_0,M_1,\varphi,\psi)\longmapsto (M_1,M_0,\psi,\varphi)$$
reflecting its underlying $\mathbb Z/2\mathbb Z$-graded structure.

\subsection{Maximal Cohen-Macaulay modules} Let $R$ be as above and assume that  it is CM.

\begin{defn} A finitely generated $R$-module $M$ is called MCM if
$\operatorname{depth}_R(M)=\dim R$. It is called indecomposable if $M\simeq M_1\oplus M_2$ implies that either $M_1=0$ or $M_2=0$.
The ring $R$ is said to be of finite CM type if it admits only finitely many isomorphism classes of indecomposable maximal CM modules.
\end{defn}

\noindent One of the fundamental sources of MCM modules is provided by syzygies.

\begin{prop}
Let 
$$
M_{k-1}\longrightarrow M_{k-2}\longrightarrow \cdots \longrightarrow M_0$$
be an exact complex of finitely generated free $R$-modules. If
$k\ge\dim R$, then the syzygy
$\textnormal{ker}(M_{k-1}\longrightarrow M_{k-2})$ is a MCM $R$-module.
\end{prop}

\noindent  For hypersurface rings, MCM modules are completely
described by matrix factorizations. If $(\varphi,\psi)$ is a matrix factorization of $f$, then
$f\operatorname{coker}(\varphi)=0$, so $\operatorname{coker}(\varphi)$ naturally carries the structure of an
$R$-module. Reducing modulo $f$ yields the $2$-periodic free resolution
$$
\cdots \xrightarrow{\varphi} R^n \xrightarrow{\psi} R^n \xrightarrow{\varphi} R^n \xrightarrow{\psi}
R^n \longrightarrow \operatorname{coker}(\varphi) \longrightarrow 0$$
and hence defines the functor
\begin{eqnarray*}
\operatorname{coker}\colon \operatorname{MF}(S,f)& \rightarrow & \operatorname{CM}(R)\\
  \qquad \qquad (\varphi,\psi) & \mapsto & \operatorname{coker}(\varphi).
\end{eqnarray*}

\begin{cor}\cite[Corollary 6.3]{Eisenbud}\label{eisenbud}
The functor $\textnormal{coker}$  
induces a bijection between the equivalence classes of reduced matrix factorizations of $f$ and the isomorphism classes of maximal CM $A$-modules having no nonzero free direct summand.\end{cor}

\noindent  When $f$ is irreducible and $(\varphi,\psi)$ is a matrix factorization of size $n$, there exist
a unit $u\in S^*$ and an integer $k$ with $0\le k\leq n$ such that
$$\det(\varphi)=uf^k, \qquad \det(\psi)=u^{-1}f^{n-k}.$$
This gives that $\operatorname{rank}_R(\operatorname{coker}(\varphi))=k$ and $\operatorname{rank}_R(\operatorname{coker}(\psi))=n-k$. Moreover,
$$
\operatorname{coker}(\varphi)\cong R^n/\operatorname{im}(\varphi)\cong \operatorname{im}(\psi).$$


\subsection{Quotient surface singularities} Let $G\subset \textnormal{GL}(n,\mathbb C)$ be a finite subgroup acting linearly on $\mathbb C^n$. The quotient $X:=\mathbb C^n/G$ is an affine variety with coordinate ring $R=\mathbb C[x_1,\ldots, x_n]^G$. 
If $G$ is small, that is, if it contains no pseudo-reflections, the image of the origin defines a quotient singularity. By Prill's theorem \cite{Prill}, quotient singularities are, up to analytic isomorphism, obtained in this way. The invariant ring $R$ is a normal domain and, by the Hochster-Eagon theorem \cite{HochsterEagon}, it is CM. Here we consider the two-dimensional case. Thus $R=\mathbb C[x,y]^G$ where $G\subset \textnormal{GL}(2,\mathbb C)$ is a finite small subgroup. Brieskorn \cite{Brieskorn} classified the finite small subgroups of $\textnormal{GL}(2,\mathbb C)$.
He proved that every finite small subgroup of $\textnormal{GL}(2,\mathbb C)$ is conjugate to a group belonging to one of the following families:
$$\mathrm C_{n,q}, \quad (\mathbb Z_{2m},\mathbb Z_{2m};\mathbb D_n,\mathbb D_n), \quad (\mathbb Z_{4m},\mathbb Z_{2m};\mathbb D_n,\mathbb C_{2n}), \quad (\mathbb Z_{2m},\mathbb Z_{2m};\mathbb O,\mathbb O), \quad (\mathbb Z_{2m},\mathbb Z_{2m};\mathbb I,\mathbb I)$$
subject to the corresponding numerical conditions given by Brieskorn in \cite{Brieskorn}. Here 
\begin{equation}\label{def-cnq} \mathcal C_{n,q}=\left\langle \left.
\begin{pmatrix}
\zeta & 0\\
0 & \zeta^q
\end{pmatrix} \right|  \  0<q<n, \ (n,q)=1\right\rangle \subset \textnormal{GL}(2,\mathbb C)\end{equation}
with $\zeta$ as a primitive $n$-th root of unity. We use $\mathbb D_n$, $\mathbb O$ and $\mathbb I$ to denote the binary dihedral, octahedral and icosahedral groups, respectively. The notation $(H_1,H_2;K_1,K_2)$ is that used by Brieskorn in \cite{Brieskorn} where the author also described the weighted dual graphs of the minimal resolutions of the corresponding quotient surface singularities. Here we focus on the cyclic and dihedral families. The geometry of these quotient surface singularities is closely related to the structure of their CM modules. This connection was developed from the representation-theoretic point of view by Auslander \cite{Auslander} and, in relation with the minimal resolution, by Wunram \cite{Wunram}. In particular, Wunram introduced the notion of \textit{special} CM modules and established an important correspondence which we quote below.

\begin{defn}\label{def-spec} Let $R$ be a rational surface singularity and let $\pi\colon\widetilde X\rightarrow\operatorname{Spec}(R)$ be the minimal
resolution. A maximal CM $R$-module $M$ is called \emph{special} if
$H^1(\widetilde X,\widetilde M^\vee)=0$ where $\widetilde M$ denotes the full sheaf associated with $M$.
\end{defn}

\begin{thm}\cite{Wunram}\label{thm-wunram} Let $R=\mathbb C[x,y]^G$ where $G\subset \textnormal{GL}(2,\mathbb C)$ is a finite small subgroup. Let $\pi\colon X\rightarrow\operatorname{Spec}(R)$ be the minimal resolution with exceptional curves $E_1,\ldots,E_r$. Then there is a bijection between the exceptional curves and the isomorphism classes of non-free indecomposable special MCM $R$-modules. If $M_i$ is the special MCM module corresponding to $E_i$, then
$$
\mu_R(M_i)=2\operatorname{rank}_R(M_i).
$$
where $\mu_R(M_i)$ denotes the minimal number of generators of $M_i$ as an $R$-module. In particular, every rank-one special MCM module is minimally generated by two elements.
\end{thm}

\noindent Since quotient surface singularities are normal, their MCM modules are reflexive. Recall that for a Noetherian domain $R$, an $R$-module $M$ is called \textit{reflexive} if the canonical homomorphism
$M\longrightarrow M^{\vee\vee}$ with $M^\vee:=\operatorname{Hom}_R(M,R)$ is an isomorphism.

\noindent For our purposes, it will also be useful to have an intrinsic characterization of special CM modules as quoted in the theorem below. 

\begin{thm} \cite{IyamaWemyss} Let $R$ be a rational surface singularity. Let $M$ be a maximal CM $R$-module. Then $M$ is special if and only if $\operatorname{Ext}^1_R(M,R)=0$.
\end{thm}

\noindent We now treat the cyclic and dihedral families separately. For each family, we construct a projected hypersurface model and study its matrix factorizations in relation to the special CM modules of the original quotient singularity.

\section{Cyclic singularities of type $\mathbb A_{n,q}$}\label{sect-anq}

\noindent Let $R_{n,q}:=\mathbb C[x,y]^{\mathcal C_{n,q}}$ be ring of invariants under the action of $C_{n,q}$ on $\mathbb{C}^2$. The quotient singularity
$$
\mathbb A_{n,q}:=\operatorname{Spec}(R_{n,q})
$$
is called the cyclic singularity of type $\mathbb A_{n,q}$. Equivalently it is said to be of type \(\displaystyle \frac{1}{n}(1,q)\)
following  Hirzebruch-Jung continued fraction expansion of $n/q$ which is of the form 
\begin{equation}\label{eq-nq}\frac{n}{q}=b_1-\frac{1}{b_2-\frac{1}{{\vdots \atop b_{r-1}-\frac{1}{b_{r}}}}}\end{equation}
for some integers $b_i\geq 2$, $i=1,\ldots, r$ and $r\geq 2$. We will write $\displaystyle \frac{n}{q}=[b_1,\ldots,b_r]$ for short to emphasize the sequence $b_1,\ldots,b_r$ in the expansion.

\noindent The structure of $R_{n,q}$ and the geometry of the minimal resolution of $\mathbb A_{n,q}$ are encoded by the
Hirzebruch-Jung continued fraction expansion of $n/q$. Furthermore, two important cases arise depending on whether $\mathcal C_{n,q}$ is a subset of $\textnormal{SL}(2,\mathbb C)$ (Gorenstein  where $q=n-1$) or not (non-Gorenstein where $q\neq n-1$). We will describe both cases starting with the Gorenstein case which recovers the classical \textit{Kleinian singularity of type $A_{n-1}$}. Later on, we will refer to this case to make comparisons for the non-Gorenstein case.

\subsection{The Gorenstein case}\label{sect-gor}

\noindent The cyclic singularity $\mathbb A_{n,q}$ is Gorenstein when $q=n-1$. In this case, we have 
$$
\mathcal C_{n,n-1}=\left\langle
\begin{pmatrix}
\varepsilon & 0\\
0 & \varepsilon^{-1}
\end{pmatrix}
\right\rangle \subset \textnormal{SL}(2,\mathbb C)$$
The singularity $\mathbb A_{n,n-1}$ is called the Kleinian singularity of type $A_{n-1}$ and its invariant ring is generated by
$u=x^n$, $v=y^n$ and $w=xy$ with the single relation $uv=w^n$. Thus we get 
$$
R_{n,n-1}\simeq {\mathbb C[u,v,w]}/{(uv-w^n)},
$$
which, after a linear change of coordinates, takes the classical form
$Z^n+X^2+Y^2=0$. The corresponding Hirzebruch-Jung expansion is

\begin{equation}\frac{n}{n-1}=[\underbrace{2,\ldots,2}_{n-1}] \end{equation}
which says that the exceptional divisor of the minimal resolution is a chain of $n-1$ smooth rational $(-2)$-curves. Hence the dual graph of the minimal resolution is of the form depicted in Figure \ref{graf-An}.

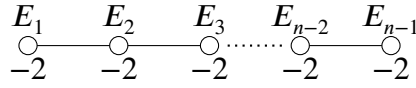
\begin{figure}[htbp]
\begin{center}
\begin{tikzpicture}[domain=-1.25:1.25,scale=1.2]
\draw (-2,0) circle [radius=0.1] node[above]{$E_1$} node[below]{$-2$};
\draw[-] (-1.9,0) -- (-1.1,0) ;
\draw (-1,0) circle [radius=0.1] node[above]{$E_2$} node[below]{$-2$};
\draw[-] (-0.9,0) -- (-0.1,0) ;
\draw (0,0) circle [radius=0.1] node[above]{$E_3$} node[below]{$-2$};
\draw[dotted,thick] (0.2,0) -- (0.8,0) ;
\draw (1,0) circle [radius=0.1] node[above]{$E_{n-2}$} node[below]{$-2$};
\draw[-] (1.1,0) -- (1.9,0) ;
\draw (2,0) circle [radius=0.1] node[above]{$E_{n-1}$} node[below]{$-2$};
 \end{tikzpicture}
\caption{Dual resolution graph of $A_{n-1}$ singularity}
\label{graf-An}
\end{center}
\end{figure}

\noindent The McKay correspondence allows us to describe explicitly the MCM modules associated with these exceptional curves. For each $0\leq t\leq n-1$, consider the $R_{n,n-1}$-modules
$$M_t=\left\{\left.\sum_{i,j\geq 0}a_{ij} x^iy^j \right| \ i-j\equiv t\pmod n \right\}$$
By \cite{Auslander}, they form a complete set of pairwise non-isomorphic indecomposable MCM $R_{n,n-1}$-modules, with $M_0=R_{n,n-1}$. The classical McKay correspondence identifies the non-free modules $M_1,\ldots,M_{n-1}$ with the exceptional curves $E_1,\ldots,E_{n-1}$, respectively. Since $R_{n,n-1}$ is Gorenstein,
every MCM $R_{n,n-1}$-module is special.   In particular, $M_1,\ldots,M_{n-1}$ are precisely the non-free indecomposable special
CM modules. The congruence condition in $M_t$ implies that every monomial
in $M_t$ is divisible by either $x^t$ or $y^{n-t}$. Hence we have 
$$M_t=R_{n,n-1}x^t+R_{n,n-1}y^{n-t}
$$
Multiplication by $x^{n-t}$ identifies $M_t$ with the ideal generated by $x^n$ and $x^{n-t}y^{n-t}=(xy)^{n-t}$. Thus, in terms of the invariant coordinates $u=x^n$, $v=y^n$, and $w=xy$, we have
$$
M_t\cong (u,w^{n-t})$$
Using the relation $uv=w^n$, this is also isomorphic to $(v, w^t)$. So, $M_t\cong (u, w^{n-t})\cong(v, w^t)$. 
Since the dual of the character $\varepsilon^t$ is $\varepsilon^{-t}=\varepsilon^{n-t}$, it follows that
$M_t^\vee\cong M_{n-t}$.

\subsubsection{Matrix factorizations} The relation $uv=w^n$ admits the following family of matrix
factorizations $(A_t,B_t)$. For each $1\le t\le n-1$, define
$$A_t=\begin{pmatrix}
v&-w^t\\
-w^{n-t}&u
\end{pmatrix},
\qquad B_t=\begin{pmatrix}
u&w^t\\
w^{n-t}&v
\end{pmatrix}$$
By a direct computation, we obtain $A_tB_t=B_tA_t=(uv-w^n)\textnormal{Id}_2$ where $\textnormal{Id}_2$ is the $2\times 2$ identity matrix over $R_{n,n-1}$.

\subsubsection{Rank-one modules} We now describe the modules $M_1,\ldots,M_{n-1}$ explicitly in terms of
matrix factorizations $(A_t,B_t)$ of the hypersurface $uv-w^n=0$.

\begin{prop} For every $1\leq t\leq n-1$, we have
$$
\operatorname{coker}(A_t)\cong (u,w^{n-t})\cong M_t.$$
Moreover, $M_t$ admits the $2$-periodic free resolution
$$
\cdots\longrightarrow R_{n,n-1}^2\xrightarrow{B_t}R_{n,n-1}^2\xrightarrow{A_t}R_{n,n-1}^2
\xrightarrow{B_t}R_{n,n-1}^2\xrightarrow{A_t}R_{n,n-1}^2\longrightarrow M_t\longrightarrow 0.$$
\end{prop}

\begin{proof} The columns of $A_t$ give the relations
$$
vu-w^tw^{n-t}=uv-w^n=0,\qquad -w^{n-t}u+uw^{n-t}=0$$
among the generators $u$ and $w^{n-t}$ of the ideal $(u,w^{n-t})$. Hence $A_t$ is a presentation matrix of
$(u,w^{n-t})$, and therefore $$
\operatorname{coker}(A_t)\cong(u,w^{n-t})\cong M_t$$
The $2$-periodic free resolution is the one associated with the matrix factorization $(A_t,B_t)$ by Eisenbud's correspondence (Corollary \ref{eisenbud}).
\end{proof}

\noindent Thus, in the Gorenstein cyclic case, the McKay and Wunram correspondences, together with Eisenbud's theory of matrix
factorizations, relate the four descriptions

$$
\begin{array}{ccc}
\{\text{non-trivial irreducible representations of }\mathcal C_{n,n-1}\}&\leftrightarrow&\{M_1,\ldots,M_{n-1}\}\\
& & \updownarrow \\
\qquad \qquad \qquad \qquad \{E_1,\ldots,E_{n-1}\} & \leftrightarrow &\{(A_1,B_1),\ldots,(A_{n-1},B_{n-1})\} 
\end{array}
$$

\subsection{The non-Gorenstein case}

\noindent We now turn to the case $q\neq n-1$. Consider the subgroup \(\mathcal C_{n,q}\) defined by (\ref{def-cnq}). In this case, $\mathcal C_{n,q}\not\subset \textnormal{SL}(2,\mathbb C)$, and the cyclic quotient singularity
$\mathbb A_{n,q}$ 
is non-Gorenstein. Unlike the Gorenstein case considered above, the invariant ring $R_{n,q}$ is in general not a hypersurface. 

\noindent We use Riemenschneider's description of the invariant ring in terms of generators and relations to obtain a suitable hypersurface in $\mathbb{C}^3$.

\begin{thm} \cite[\S 2.A.]{riemen} \label{riemen}
With the notation above, let us assume that
$$
\frac{n}{n-q} =a_2-\frac{1}{a_3-\frac{1}{{\vdots \atop a_{e-1}-\frac{1}{a_{e-1}}}}}=[a_2,\ldots,a_{e-1}]$$
where $a_t\geq2$ and $e$ is the embedding dimension of $R_{n,q}$. Then $R_{n,q}$ is generated by the $e$ monomials
${u_1}^{i_t}{u_2}^{j_t}$ with $0\leq t\leq e-1$ where
\begin{align*}
(i_0,j_0)&=(n,0),\\
(i_1,j_1)&=(n-q,1),\\
(i_t,j_t)&=(a_ti_{t-1}-i_{t-2}, a_tj_{t-1}-j_{t-2}), \qquad 2\leq t\leq e-1\\
(i_{e-1}, j_{e-1})&=(0,n).
\end{align*}

\end{thm}

\noindent Let $z_1=x^n$, $z_2=x^{n-q}y$, $z_3=x^{i_2}y^{j_2}$, $\ldots $, $z_e=y^n$ denote the monomial generators of $R_{n,q}$ given by Theorem \ref{riemen}. Among the defining relations of $R_{n,q}$ are the equations
$$z_{t-1}z_{t+1}=z_t^{a_t},\qquad 2\le t\le e-1.$$
We project the affine embedding of $\mathbb A_{n,q}=\operatorname{Spec}(R_{n,q})$ onto the three coordinates $z_1$, $z_2$, and $z_e$. By setting $x=z_1$, $z=z_2$ and $y=z_e$, we identify the image of this projection with the hypersurface 
\begin{equation}\label{eq-anq}
\overline R_{n,q}={\mathbb C[x,y,z]}/{(z^n-x^{n-q}y)} .
\end{equation} 

The minimal resolution of $\operatorname{Spec}\left(\overline R_{n,q}\right)$ has exceptional divisor shown in Figure \ref{graf-Anq}.

\begin{figure}[h]
\begin{center}
\begin{tikzpicture}[domain=-1.25:1.25,scale=1.2]
\draw (-2,0) circle [radius=0.1] node[below]{$-b_1$};
\draw[-] (-1.9,0) -- (-1.1,0) ;
\draw (-1,0) circle [radius=0.1]   node[below]{$-b_2$};
\draw[-] (-0.9,0) -- (-0.1,0) ;
\draw (0,0) circle [radius=0.1]  node[below]{$-b_3$};
\draw[dotted,thick] (0.2,0) -- (0.8,0) ;
\draw (1,0) circle [radius=0.1]  node[below]{$-b_{r-1}$};
\draw[-] (1.1,0) -- (1.9,0) ;
\draw (2,0) circle [radius=0.1]  node[below]{$-b_r$};
 \end{tikzpicture}
\caption{Dual resolution graph of $\mathbb{A}_{n,q}$ singularity}
\label{graf-Anq}
\end{center}
\end{figure}
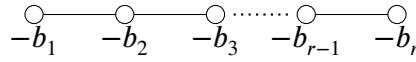

Here the integers $b_1,\ldots,b_r$ are determined by the Hirzebruch-Jung expansion of $\displaystyle \frac{n}{q}$ given by (\ref{eq-nq}).

\begin{rem}\label{rem-anq-oka} We note that one can recover the dual resolution graph in Figure \ref{graf-Anq} using Oka's method (\cite{oka1,oka2}) whose main ingredients are three dimensional cones, Newton polygons and their \textit{appropriate} subdivisions in three dimensional space. For our equation, we first need to apply the linear transformation \((x,y,z)\mapsto (x,x+y,z)\) on \(\overline R_{n,q}\) and obtain \(z^n-x^{n-q}y-x^{n-q+1}=0\) so that there are sufficient number of monomials to produce a cone of the right dimension. Additionally, blowing down $(-1)$-curves in the output of Oka's method is required to match the graph in question.
\end{rem}


\subsubsection{Matrix factorizations for the projected hypersurface}  The advantage of working with $\overline R_{n,q}$ is that it is a hypersurface, so its MCM modules can be studied by means of matrix factorizations. Our aim is not to describe all MCM modules over $\overline R_{n,q}$, but rather to construct a finite family whose
members recover the special CM modules of the original quotient singularity $R_{n,q}$ after extension of scalars and taking reflexive
hulls. We first recall which modules over $R_{n,q}$ are special.

As in the Gorenstein case, the indecomposable MCM modules over $R_{n,q}$ are described by the semi-invariant modules. For
$0\leq t\leq n-1$, let
$$
M_t=\{
h\in\mathbb C[x,y]\mid
g\cdot h=\varepsilon^t h \}$$
where $\varepsilon$ is a fixed primitive $n$-th root of unity.
Equivalently,
$$
M_t=\left\{ \left. \sum a_{ij}x^iy^j\right|   i+qj\equiv t\pmod n \right \}.$$
The modules $M_0,\ldots, M_{n-1}$ form a complete set of pairwise non-isomorphic indecomposable MCM $R_{n,q}$-modules. In contrast to the Gorenstein case, only some of these modules are special. They are described explicitly as follows.

\begin{thm} \cite{WemyssA} The indecomposable non-free special CM $R_{n,q}$-modules are
$M_{i_1}, M_{i_2}, \ldots, M_{i_{e-2}}$ where the integers $i_t$ and $j_t$ are those introduced above. Moreover,
$$
M_{i_t}= R_{n,q}x^{i_t}+R_{n,q}y^{j_t}, \qquad 1\leq t\leq e-2.$$
\end{thm}

The following characterization will be useful when comparing the modules obtained from matrix factorizations with the special CM modules over $R_{n,q}$.

\begin{prop} \cite{IyamaWemyss,WemyssA}
Let $M$ be an indecomposable MCM $R_{n,q}$-module. Then the following conditions are equivalent:

\begin{enumerate}
\item[(i)] $M$ is special, 

\item[(ii)] $\operatorname{Ext}^1_{R_{n,q}}(M,R_{n,q})=0$, 

\item[(iii)] $M\cong M_{i_t}$ for some $1\leq t\leq e-2$.
\end{enumerate}
\end{prop}

\noindent The equivalence of {\rm (i)} and {\rm (ii)} follows from
\cite{IyamaWemyss}, while the equivalence of {\rm (i)} and {\rm (iii)}
follows from \cite{WemyssA}.

We consider the hypersurface ring $\overline{R}_{n,q}$ given by equation (\ref{eq-anq}), and construct the matrix
factorizations that will be used to recover the special CM modules described above. 
For each $1\leq k\leq n-1$, define
$$
A_k=\begin{pmatrix}
y & -z^k\\
-z^{n-k} & x^{n-q}
\end{pmatrix},
\qquad
B_k=
\begin{pmatrix}
x^{n-q} & z^k\\
z^{n-k} & y
\end{pmatrix}
$$

\begin{prop}\label{cyclic-mf}
For every $1\leq k\leq n-1$, the pair $(A_k,B_k)$ is a $2\times 2$ matrix factorization of the defining equation
$f=x^{n-q}y-z^n$. 
\end{prop}

\begin{proof}
By a direct computation, we have $A_kB_k=B_kA_k=(x^{n-q}y-z^n)\textnormal{Id}_2$. 
\end{proof}

\subsubsection{Rank-one modules} Let $N_k:=\operatorname{coker}(A_k)$. By Eisenbud's correspondence (Corollary \ref{eisenbud}), $N_k$ is a MCM $\overline R_{n,q}$-module for every $1\leq k\leq n-1$. The following result describes these modules explicitly.

\begin{prop}\label{cyclic-rankone}
For every $1\leq k\leq n-1$, the module $N_k$ is a non-free indecomposable rank-one maximal 
CM $\overline R_{n,q}$-module. Moreover, $N_k$ is minimally generated by two
elements and admits the $2$-periodic free resolution
$$
\cdots \longrightarrow \overline R_{n,q}^{2}\xrightarrow{B_k} \overline R_{n,q}^{2}
\xrightarrow{A_k}\overline R_{n,q}^{2} \xrightarrow{B_k} \overline R_{n,q}^{2}\xrightarrow{A_k}
\overline R_{n,q}^{2}\longrightarrow N_k\longrightarrow 0.
$$
\end{prop}

\begin{proof} Since $(A_k,B_k)$ is a matrix factorization of $f$, Eisenbud's
correspondence implies that $N_k=\operatorname{coker}(A_k)$ is a MCM $\overline R_{n,q}$-module and that the associated $2$-periodic complex is exact. The matrix factorization is reduced because
all entries of $A_k$ and $B_k$ belong to the maximal ideal $(x,y,z)$. Hence the presentation of $N_k$ is minimal, so $N_k$ is minimally generated by two elements. To compute the rank, observe that
$\textnormal{det}(A_k)=x^{n-q}y-z^n=f$. Since $f=0$ in $\overline R_{n,q}$ and $A_k$ is nonzero over the fraction
field of $\overline R_{n,q}$, the matrix $A_k$ has rank one. Hence $\operatorname{rank}_{\overline R_{n,q}}(N_k)=1$. Since $N_k$ has rank one but is minimally generated by two elements, it is not free.

\noindent It remains to prove that $N_k$ is indecomposable. Let us assume that 
$N_k\simeq N'\oplus N''$. Since $\overline R_{n,q}$ is a two-dimensional domain and $N_k$ is MCM, $N_k$ is torsion-free. Hence its direct summands $N'$ and $N''$ are torsion-free. Every nonzero
torsion-free module over a domain has positive rank. Therefore,
$$
1=\operatorname{rank}(N_k)=\operatorname{rank}(N')+\operatorname{rank}(N'')$$
It follows that one of the two summands has rank zero, and hence must be
zero. Thus $N_k$ is indecomposable.\end{proof}

\noindent The projected hypersurface is naturally related to the invariant ring through the injective homomorphism
\begin{eqnarray}\label{eq-iota} \iota \colon \overline R_{n,q} &\rightarrow &R_{n,q} 
\\
(x,y,z) &\mapsto & (x^n, y^n, x^{n-q}y )\nonumber \end{eqnarray}
Thus every $\overline R_{n,q}$-module determines an $R_{n,q}$-module by the extension of scalars. As $R_{n,q}$ is normal, we may pass to the reflexive hull.

\subsubsection{Comparison with the special modules} Recall that the indecomposable special MCM modules
over $R_{n,q}$ are $M_{i_1}, M_{i_2},\ldots, M_{i_{e-2}}$ where the integers $i_t$ are defined by Theorem \ref{riemen}. The following theorem compares the modules coming from the matrix factorizations $(A_k,B_k)$ with the special CM modules over $R_{n,q}$.

\begin{thm}\label{cyclic-comparison}
For every $1\leq k\leq n-1$, the following hold.

\begin{enumerate}

\item[(i)] There is an isomorphism of $\overline R_{n,q}$-modules $N_k\cong (z^{n-k},y)$. 

\item[(ii)] The reflexive hull of the $R_{n,q}$-module $N_k\otimes_{\overline R_{n,q}}R_{n,q}$ satisfies
$$(N_k\otimes_{\overline R_{n,q}}R_{n,q})^{\vee\vee}\cong M_{\lambda_k}$$
where $\lambda_k\equiv qk\pmod n$ and $1\leq \lambda_k \leq n-1$. 

\item[(iii)] The module $(N_k\otimes_{\overline R_{n,q}}R_{n,q})^{\vee\vee}$ is special if and only if
$\lambda_k \in{i_1,\ldots, i_{e-2}}$. In this case, we have 
$$(N_k\otimes_{\overline R_{n,q}}R_{n,q})^{\vee\vee}\cong M_{i_t}$$
where $1\leq t\leq e-2$ is the unique index such that $\lambda_k=i_t$.

\item[(iv)] Every indecomposable special MCM $R_{n,q}$-module is obtained in this way. More precisely, for every
$1\leq t\leq e-2$ there exists a unique $1\leq k_t\leq n-1$ such that
$$
qk_t\equiv i_t\pmod n$$
and 
$$(N_{k_t}\otimes_{\overline R_{n,q}}R_{n,q})^{\vee\vee}\cong M_{i_t}$$

\end{enumerate}
\end{thm}

\begin{proof} (i) Consider the ideal $I_k=(z^{n-k}, y)\subset \overline R_{n,q}$. Let us define
\begin{eqnarray*} 
\varphi\colon\overline R_{n,q}^2 &\rightarrow & I_k, \\
 (a,b)&\mapsto &az^{n-k}+by. \end{eqnarray*}
The columns of
$$
A_k=\begin{pmatrix}
y&-z^k\\
-z^{n-k}&x^{n-q}
\end{pmatrix}
$$
belong to $\textnormal{ker}(\varphi)$ since $yz^{n-k}-z^{n-k}y=0$ and
$-z^kz^{n-k}+x^{n-q}y=-z^n+x^{n-q}y=0$ in $\overline R_{n,q}$. We claim that they generate $\ker(\varphi)$. Let $(a,b)\in \ker(\varphi)$. Choose $\widetilde a,\widetilde b\in \mathbb C[x,y,z]$ which are not identically zero, and assume $\varphi(\widetilde a,\widetilde b)=0$. Then
$$\widetilde a z^{n-k}+\widetilde b y\in (x^{n-q}y-z^n).$$
So, there exists an element $c\in\mathbb C[x,y,z]$ such that
$$
\widetilde a z^{n-k}+\widetilde b y=c(x^{n-q}y-z^n).$$
Hence
$$(\widetilde a+cz^k)z^{n-k}+(\widetilde b-cx^{n-q})y=0$$
Since $y$ and $z^{n-k}$ are relatively prime in $\mathbb C[x,y,z]$, there exists $d\in\mathbb C[x,y,z]$ such that
$\widetilde a+cz^k=dy$ and $\widetilde b-cx^{n-q}=-dz^{n-k}$. Therefore, modulo $(x^{n-q}y-z^n)$, we get 
$$
\binom{a}{b}=d\binom{y}{-z^{n-k}}+c\binom{-z^k}{x^{n-q}}$$
So $\ker(\varphi)$ is generated by the columns of $A_k$. Hence $A_k$ is a presentation matrix of $I_k$ and we have
$$N_k=\operatorname{coker}(A_k)\cong (z^{n-k},y).$$

\noindent (ii) After an extension of scalars to $R_{n,q}$, we obtain
$$N_k\otimes_{\overline R_{n,q}}R_{n,q}\cong (z^{n-k},y)\otimes_{\overline R_{n,q}}R_{n,q}.$$
Under the inclusion $\iota$ defined by (\ref{eq-iota}), the generators are transformed into
$$
z^{n-k}\longmapsto x^{(n-q)(n-k)}y^{n-k},\qquad y\longmapsto y^n.$$
Multiplying by $y^{-(n-k)}$ gives the equivalent fractional ideal
$(x^{(n-q)(n-k)},y^k)$. The two generators transform according to the same character of $C_{n,q}$ since
$$(n-q)(n-k)\equiv qk\pmod n$$
and $y^k$ also has character $qk$ modulo $n$. So, this fractional ideal has character
$$
\lambda_k\equiv qk\pmod n$$
By the classification of rank-one reflexive $R_{n,q}$-modules by characters, its reflexive hull is the semi-invariant module $M_{\lambda_k}$. Hence we have
$$(N_k\otimes_{\overline R_{n,q}}R_{n,q})^{\vee\vee}\cong M_{\lambda_k}.$$

\noindent  (iii) Wunram's classification says that the indecomposable special MCM $R_{n,q}$-modules are $M_{i_1},M_{i_2},\ldots, M_{i_{e-2}}$. Therefore $M_{\lambda_k}$ is special if and only if
$\lambda_k\in{i_1,\ldots, i_{e-2}}$. In this case there is a unique $t$ with $1\leq t\leq e-2$ such that
$\lambda_k=i_t$. So we get 
$$(N_k\otimes_{\overline R_{n,q}}R_{n,q})^{\vee\vee}\cong M_{i_t}.$$

\noindent (iv) Since $(n,q)=1$, multiplication by $q$ induces a permutation of the nonzero residue classes modulo $n$. Hence, for every $i_t$, there exists a unique integer $1\leq k_t\leq n-1$ such that
$$qk_t\equiv i_t\pmod n$$
Applying (ii) gives
$$(N_{k_t}\otimes_{\overline R_{n,q}}R_{n,q})^{\vee\vee}\cong M_{i_t}.$$
So, we conclude that every indecomposable special MCM $R_{n,q}$-module appears as the reflexive hull of the extension of one of the matrix-factorization modules $N_k$. \end{proof}

\noindent Note that the matrix factorizations above need not exhaust all matrix factorizations of the non-isolated hypersurface $\textnormal{Spec}(\overline R_{n,q})$. However, the finite family $(A_k,B_k)$ with $1\leq k\leq n-1$ is sufficient to recover all indecomposable special MCM modules over the original quotient singularity $\textnormal{Spec}(R_{n,q})$ after the extension of scalars and taking reflexive hulls. The following example demonstrates that there are matrix factorizations outside the family $(A_k,B_k)$, and that distinct MCM modules over $\overline R_{n,q}$ may determine the same reflexive module over $R_{n,q}$.

\begin{ex} \label{ex1} Assume $r=n-q\geq 2$. Consider
$$
C=\begin{pmatrix}
xy&-z\\
-z^{n-1}&x^{r-1}
\end{pmatrix}, \qquad D=\begin{pmatrix}
x^{r-1}&z\\
z^{n-1}&xy
\end{pmatrix}
$$
We have $CD=DC=(x^ry-z^n)\textnormal{Id}_2=f\textnormal{Id}_2$, so $(C,D)$ is a $2\times2$ matrix factorization of $f$. 
Moreover, we have
$$\operatorname{Fitt}_1(\operatorname{coker}(C))=(xy, z, x^{r-1}), \qquad \operatorname{Fitt}_1(N_k)=(y, x^r, z^{\min\{k,n-k\}})$$
for every $1\leq k\leq n-1$. These ideals are distinct: indeed,
$x^{r-1}\in\operatorname{Fitt}_1(\operatorname{coker}(C))$ and $x^{r-1}\notin\operatorname{Fitt}_1(N_k)$, since reducing modulo $(y,z)$ gives $\operatorname{Fitt}_1(N_k)/(y,z)=(x^r)$. Hence, since Fitting ideals are invariant under isomorphism, we obtain $\operatorname{coker}(C)\not\cong N_k$ for every $1\leq k\leq n-1$. On the other hand, after the extension of scalars to $R_{n,q}$, the ideal $(x^{r-1},z)$ becomes $(x^{n(r-1)},x^ry)$. Multiplying by the fractional monomial $x^{-r}$ gives the equivalent fractional ideal
$(x^{n(r-1)-r},y)$. Since $n(r-1)-r\equiv q\pmod n$ both generators have character $q$. Therefore we get  $(L\otimes_{\overline R_{n,q}}R_{n,q})^{\vee\vee}\cong M_q$. Because
we have $(N_1\otimes_{\overline R_{n,q}}R_{n,q})^{\vee\vee}\cong M_q$, we see that the modules $L$ and $N_1$ are non-isomorphic over $\overline R_{n,q}$ but they determine the same reflexive module over $R_{n,q}$. Furthermore, note that this module is special when $q\in{i_1,\ldots, i_{e-2}}$. 
\end{ex}

\noindent We now turn to the dihedral case where the geometry is richer resulting some special CM modules of rank two.
\section{Dihedral singularities of type $\mathbb{D}_{n,q}$}\label{sect-dnq}

\noindent Let $n$ and $q$ be relatively prime positive integers with $1<q<n$. Put $m=n-q$. Following \cite{riemen}, we denote by $\mathcal D_{n,q}$ the dihedral subgroup of $\textnormal{GL}(2,\mathbb C)$ corresponding to these numerical data.
 More precisely, if $m$ is odd, then
$$
\mathcal D_{n,q}=(C_{2m},C_{2m}; BD_{2q},BD_{2q})$$
where $BD_{2q}$ denotes the binary dihedral group of order $4q$. If $m$ is even, write $m=2^{k-2}\ell$ with $\ell$ odd. In this case,
$$
\mathcal D_{n,q}=(C_{4m},C_{2m};BD_{2q},C_{2q})$$
In both cases $\mathcal D_{n,q}$ is a finite small subgroup of $\textnormal{GL}(2,\mathbb C)$. We denote the corresponding quotient surface singularity by
$$
\mathbb D_{n,q}=\mathbb C^2/\mathcal D_{n,q}=\operatorname{Spec}(R_{n,q}),\qquad
R_{n,q}=\mathbb C[x,y]^{\mathcal D_{n,q}}
$$
The geometry of the minimal resolution is determined by the Hirzebruch-Jung expansion
$$
\frac{n}{q}=[b_3,\ldots,b_r]=b_3-\frac{1}{b_4-\frac{1}{{\vdots \atop b_{r-1}-\frac{1}{b_r}}}}, \qquad b_i\geq 2$$
The corresponding weighted dual graph of the minimal resolution is sketched in Figure \ref{graf-Dnq}.

\begin{figure}[htbp]
\begin{center}
\begin{tikzpicture}[domain=-1.25:1.25,scale=1.2]
\draw (-2,0) circle [radius=0.1] node[above]{$-2$};
\draw[-] (-1.9,0) -- (-1.1,0) ;
\draw (-1,0) circle [radius=0.1] node[above]{$-b_3$};
\draw[-] (-0.9,0) -- (-0.1,0) ;
\draw (0,0) circle [radius=0.1] node[above]{$-b_4$};
\draw[dotted,thick] (0.2,0) -- (0.8,0) ;
\draw (1,0) circle [radius=0.1] node[above]{-$b_{r-1}$};
\draw[-] (1.1,0) -- (1.9,0) ;
\draw (2,0) circle [radius=0.1] node[above]{$-b_r$};
\draw[-] (-1,-0.1) -- (-1,-0.9);
\draw (-1,-1) circle [radius=0.1] node[right]{$-2$};
 \end{tikzpicture}
\caption{Dual resolution graph of $\mathbb{D}_{n,q}$ singularity}
\label{graf-Dnq}
\end{center}
\end{figure}
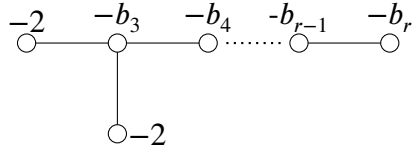

\noindent The long branch of the resolution graph is determined by the Hirzebruch-Jung continued fraction, while the two short branches consist of exceptional curves of self-intersection $-2$. So the dihedral resolution graph can be viewed as the cyclic chain together with two
additional $(-2)$-vertices attached to its first vertex. This allows much of the continued fraction notation from the cyclic case to be carried
over to the dihedral setting. In particular, the expansion
$$
\frac{n}{q}=[b_3,b_4,\ldots,b_r]
$$
determines the corresponding $i$- and $j$-series. Its dual continued
fraction
\begin{equation}\label{eq-n-nq}
\frac{n}{n-q}=[a_2,a_3,\ldots,a_{e-1}]=a_2-\frac{1}{a_3-\frac{1}{{\vdots \atop a_{e-2}-\frac{1}{a_{e-1}}}}},\qquad a_i\geq 2\end{equation}
where $e=3+\sum_{i=3}^r(b_i-2)$. The integers $a_2,a_3,\ldots, a_{e-1}$ in (\ref{eq-n-nq}) give rise to the $r$-, $c$-, and $d$-series used in  Riemenschneider's description of explicit generators of $R_{n,q}$ which we quote below.

\begin{thm} \cite{riemen}
Let
$$
v_1:=xy,\qquad
v_2:=x^{2q}+(-1)^{a_2}y^{2q},
\qquad
v_3:=x^{2q}+(-1)^{a_2-1}y^{2q}.
$$
Then the invariant ring $R_{n,q}=\mathbb C[x,y]^{\mathcal D_{n,q}}$
is generated by the invariant polynomials
$$
z_1=v_1^{2(n-q)} \qquad z_j=v_1^{r_j}v_2^{c_j}v_3^{d_j},\qquad 2\le j\le e$$
where the $r$-, $c$-, and $d$-series are defined by
\begin{align*}
c_2&=1, & d_2&=0, & r_2&=a_2(n-q)-q,\\
c_3&=0, & d_3&=1, & r_3&=r_2-(n-q),\\
c_4&=1, & d_4&=a_3-1, & r_4&=(a_3+1)r_3-r_2,
\end{align*}
and, for $5\le j\le e$,
\begin{align*}
c_j&=a_{j-1}c_{j-1}-c_{j-2},\\
d_j&=a_{j-1}d_{j-1}-d_{j-2},\\
r_j&=a_{j-1}r_{j-1}-r_{j-2}.
\end{align*}
\end{thm}
With respect to these generators,
\begin{cor}\cite{riemen}  The quotient singularity 
$\mathcal{D}_{n,q}=\operatorname{Spec}(R_{n,q})$ is defined by the following relations:
\begin{align*}
z_2^2
&=z_1\bigl(z_3^2+z_1^{a_2-1}\bigr), & \\
z_1z_i&=z_2z_3^{a_3-2}\cdots z_{i-2}^{a_{i-2}-2}z_{i-1}^{a_{i-1}-1}, & 4\leq i\leq e, \qquad  \qquad  \qquad \\
z_2z_i&=z_3^{a_3-2}\cdots z_{i-2}^{a_{i-2}-2}z_{i-1}^{a_{i-1}-1}\bigl(z_3^2+z_1^{a_2-1}\bigr), & 4\leq i\leq e,  \qquad  \qquad  \qquad \\
z_{i-1}z_{i+1}&=z_i^{a_i}, &  4\leq i\leq e-1,  \qquad  \qquad  \qquad  \\
z_jz_i&=z_{j+1}^{a_{j+1}-1}z_{j+2}^{a_{j+2}-2}\cdots z_{i-2}^{a_{i-2}-2}z_{i-1}^{a_{i-1}-1}, & 4\leq j+1<i-1\leq e-1. \qquad  \qquad  \qquad
\end{align*}
\end{cor}

\noindent We now construct a hypersurface model of the dihedral quotient singularity by eliminating the variables in Riemenschneider's defining equations. Since these equations are expressed in terms of the vector
$[a]:=[a_2,\ldots, a_{e-1}]$, whereas the topology of the singularity is encoded by the Hirzebruch-Jung vector $[b]:=[b_3,\ldots,b_r]$, it is useful to relate these two descriptions. The following proposition provides an inductive description of $[a]$ in
terms of $[b]$.

\begin{prop}\label{dnq-gen} Suppose that the dual continued fraction associated with the truncated
vector $[b']=[b_4,\ldots, b_r]$ is $[a']=[a_2,a_3,\ldots, a_\ell]$. Then the dual continued fraction associated with
$[b]=[b_3,b_4,\ldots, b_r]$ is $[a]=[\underbrace{2,\ldots, 2}_{b_3-2}, a_2+1, a_3,\ldots, a_\ell]$. 
\end{prop}

\begin{proof}
Let $\displaystyle \frac{n}{q}=[b_3,b_4,\ldots, b_r]$ and set
$$
\frac{n'}{q'}=[b_4,\ldots, b_r],\qquad \frac{n''}{q''}=[b_5,\ldots, b_r].
$$
Since
$$
\frac{n'}{q'}=b_4-\frac{q''}{n''}=\frac{b_4n''-q''}{n''}$$
we obtain $q'=n''$. Therefore
$$
\frac{n}{q}=b_3-\frac{q'}{n'}=b_3-\frac{n''}{n'}$$
which implies $n=b_3n'-n''$ with $q=n'$.

\noindent We first determine the initial coefficients of the dual
continued fraction. If $b_3=2$, then $\displaystyle\frac{n}{q}=2-\frac{n''}{n'}$ and therefore
$$
\frac{n}{n-q}=\frac{2n'-n''}{n'-n''}
$$
The computation below shows that this continued fraction is
$[a_2+1,a_3,\ldots, a_\ell]$. 

\noindent Assume now that $b_3>2$. Then
$$
\frac{n}{n-q}=\frac{b_3n'-n''}{(b_3-1)n'-n''}=2-\frac{(b_3-2)n'-n''}{(b_3-1)n'-n''}$$
Repeating the same computation gives
$$
\frac{n}{n-q}=2-\frac{1}{2-\frac{(b_3-3)n'-n''}{(b_3-2)n'-n''}}
$$
After $k$ iterations, we obtain
$$
\frac{n}{n-q}=2-\frac{1}{2-\frac{1}{\ddots-\frac{1}{2-\frac{(b_3-k-1)n'-n''}{(b_3-k)n'-n''}}}}$$
where there are $k$ leading coefficients equal to $2$. The process terminates after $b_3-2$ iterations. Indeed, the remaining
fraction is $\frac{2n'-n''}{n'-n''}$ and $2n'-n''>2(n'-n'')$ for $n''>0$. Hence the first $b_3-2$ coefficients of the dual continued
fraction are equal to $2$.

\noindent It remains to determine the remaining coefficients. By assumption,
$$
\frac{n'}{n'-q'}=\frac{n'}{n'-n''}=[a_2,a_3,\ldots, a_\ell]$$
Therefore,
$$
\frac{n'}{n'-n''}=a_2-\frac{p}{n'-n''}$$
for some integer $0<p<n'-n''$. Equivalently,
$$
n'=\frac{a_2n''+p}{a_2-1}.$$
Moreover,
$$
\frac{p}{n'-n''}=\frac{(a_2-1)p}{n''+p},
$$
whose continued fraction is $[a_3,\ldots, a_\ell]$.  Substituting the above expression for $n'$ into
$\frac{2n'-n''}{n'-n''}$ gives
$$
\frac{2n'-n''}{n'-n''}=\frac{(a_2+1)n''+2p}{n''+p}=(a_2+1)-\frac{(a_2-1)p}{n''+p}.$$
Hence
$$
\frac{2n'-n''}{n'-n''}=[a_2+1,a_3,\ldots, a_\ell]$$
Combining this with the initial string of $2$'s, we conclude that 
$$
[a]=[\underbrace{2,\ldots, 2}_{b_3-2}, a_2+1, a_3,\ldots, a_\ell].$$
\end{proof}

\begin{prop}\label{dnq-4} Let $\mathbb{D}_{n,q}$ be the dihedral quotient singularity whose long
branch is determined by $[b]=[b_3,b_4]$. Then the associated dual continued fraction is
$$[a]=[\underbrace{2,\ldots, 2}_{b_3-2}, 3, \underbrace{2,\ldots,2}_{b_4-2}].
$$
\end{prop}

\begin{proof}
The continued fraction $[b_4]$ corresponds to the cyclic singularity
$\frac{b_4}{1}$, whose dual continued fraction is
$[\underbrace{2,\ldots, 2}_{b_4-2}]$. Applying Proposition \ref{dnq-gen} to the vector $[b]=[b_3,b_4]$, we obtain
$$[a]=[\underbrace{2,\ldots,2}_{b_3-2}, 3, \underbrace{2,\ldots,2}_{b_4-2}]$$
\end{proof}

\begin{rem} The preceding proposition shows the duality between the continued
fractions $[b]$ and $[a]$. More generally, suppose that
$$[b]=[b_3,2,\ldots, 2, b_r], \qquad b_3, b_r\geq 2$$
Applying Proposition \ref{dnq-gen} repeatedly we get 
$$[a]=[\underbrace{2,\ldots, 2}_{b_3-2},r-1,\underbrace{2,\ldots, 2}_{b_r-2}].$$
Conversely, the dual continued fraction associated with the latter vector
is precisely $[b]=[b_3,2,\ldots,2, b_r]$. Thus, strings of $2$'s along the long branch are transformed into a
single central coefficient in the dual continued fraction, while the
coefficients at the two ends remain unchanged.
\end{rem}

\subsection{Invariant rings and continued fractions} Our next goal is to obtain a hypersurface model birational to
$\mathbb D_{n,q}$ by eliminating the intermediate generators in Riemenschneider's presentation. Throughout this subsection, we assume
$e\geq 5$.

\begin{prop} \cite{riemen} \label{reduced-ideal}
The defining ideal of the invariant ring $R_{n,q}$ is generated by the relations
\begin{align*}
z_2^2    &=z_1\bigl(z_3^2+z_1^{a_2-1}\bigr),\\
z_1z_e    &=z_2z_3^{a_3-2}\cdots       z_{e-2}^{a_{e-2}-2}z_{e-1}^{a_{e-1}-1},\\
z_{e-2}z_e     &=z_{e-1}^{a_{e-1}},\\
z_jz_{j+3}     &=z_{j+1}^{a_{j+1}-1}z_{j+2}^{a_{j+2}-1},    \qquad 3\leq j\leq e-3
\end{align*}
\end{prop}

\noindent We first express the intermediate generators $z_2,z_3,\ldots, z_{e-2}$ rationally in terms of
$z_1,z_{e-1}$ and $z_e$. For this purpose, define two sequences $\{s_j\}$ and $\{\widehat{s}_j\}$ by
\begin{align*}
s_e&=0,
& s_{e-1}&=1,
& s_{e-2}&=a_{e-1},\\
\widehat{s}_{e-1}&=0,
& \widehat{s}_{e-2}&=1,
& \widehat{s}_{e-3}&=a_{e-2}
\end{align*}
and recursively by
\begin{align*}
s_j&=(a_{j+1}-1)s_{j+1} +(a_{j+2}-1)s_{j+2} -s_{j+3}, &&3\leq j\leq e-3,\\
\widehat{s}_j&=(a_{j+1}-1)\widehat{s}_{j+1}+(a_{j+2}-1)\widehat{s}_{j+2}-\widehat{s}_{j+3},&&3\leq j\leq e-4
\end{align*}

\noindent We also set
$\alpha =1+\sum_{j=3}^{e-1}(a_j-2)s_j$ and $\beta =1+\sum_{j=3}^{e-2}(a_j-2)\widehat{s}_j$. 

\begin{lem}\label{elimination-D} In the fraction field of $R_{n,q}$, we have
$$z_2=\frac{z_1z_e^\beta}{z_{e-1}^\alpha}$$
and \[ z_j=\frac{z_{e-1}^{s_j}}{z_e^{\widehat{s}_j}},\qquad  3\leq j\leq e-2.\]

\end{lem}

\begin{proof} The relation $z_{e-2}z_e=z_{e-1}^{a_{e-1}}$ gives
$$
z_{e-2}=\frac{z_{e-1}^{a_{e-1}}}{z_e}=\frac{z_{e-1}^{s_{e-2}}}{z_e^{\widehat{s}_{e-2}}}.$$
By taking $j=e-3$ in 
$$z_jz_{j+3}=z_{j+1}^{a_{j+1}-1}z_{j+2}^{a_{j+2}-1}$$
and, using $s_e=0$ and $\widehat{s}_e=0$ with $s_{e-1}=1$ and $\widehat{s}_{e-1}=0$, we obtain
$$z_{e-3}=\frac{z_{e-1}^{s_{e-3}}}{z_e^{\widehat{s}_{e-3}}}.$$
Now let us assume that the formulas are known for $z_{j+1}, z_{j+2}$ and $z_{j+3}$. Then
$$z_j=\frac{z_{j+1}^{a_{j+1}-1}z_{j+2}^{a_{j+2}-1}}{z_{j+3}}=\frac{z_{e-1}^{(a_{j+1}-1)s_{j+1}+(a_{j+2}-1)s_{j+2}-s_{j+3}}}{z_e^{(a_{j+1}-1)\widehat{s}_{j+1}+(a_{j+2}-1)\widehat{s}_{j+2}-\widehat{s}_{j+3}}}=\frac{z_{e-1}^{s_j}}{z_e^{\widehat{s}_j}}$$
for  $3\leq j\leq e-2$. It remains to determine $z_2$. From the second relation in Proposition \ref{reduced-ideal},
$$
z_1z_e=z_2\left(\prod_{j=3}^{e-2}z_j^{a_j-2}\right)z_{e-1}^{a_{e-1}-1}.$$
Substituting the expressions above gives
\begin{align*}
z_2&=z_1z_e^{1+\sum_{j=3}^{e-2}(a_j-2)\widehat{s}_j}z_{e-1}^{-(a_{e-1}-1+\sum_{j=3}^{e-2}(a_j-2)s_j)}
\end{align*}
By definition, we have $1+\sum_{j=3}^{e-2}(a_j-2)\widehat{s}_j=\beta $ and, since $s_{e-1}=1$, we have
$$
a_{e-1}-1+\sum_{j=3}^{e-2}(a_j-2)s_j=1+\sum_{j=3}^{e-1}(a_j-2)s_j=\alpha $$
Hence we obtain $\displaystyle z_2=\frac{z_1z_e^\beta}{z_{e-1}^\alpha}$. \end{proof}

\noindent Note that, for $e=4$, we set $s_3=1$, $\widehat s_3=0$, $\alpha=a_3-1$, and $\beta=1$. We can now eliminate the intermediate generators.

\begin{thm}\label{projected-hypersurface-D} Let $x=z_1$, $y=z_{e-1}$ and $z=z_e$. Then the projection onto these
three coordinates has the image 
$\overline X_{n,q}$ defined by the zero set $V(F)$ of
\begin{equation}\label{eq-projected-D}
F=y^{2\alpha+2s_3}+x^{a_2-1}y^{2\alpha}z^{2\widehat s_3}-xz^{2\beta+2\widehat s_3}.\end{equation}
In particular, the local ring of $\overline X_{n,q}$ at the origin is
$\overline R_{n,q}=\mathbb C[x,y,z]_{(x,y,z)}/(F)$ and the induced rational map
$\textnormal{Spec}(R_{n,q})\dashrightarrow \textnormal{Spec}\left(\overline R_{n,q}\right)$ is birational.
\end{thm}

\begin{proof} By Lemma \ref{elimination-D}, we have
$$z_2=\frac{x\,y^\alpha}{z^\beta},\qquad z_3=\frac{y^{s_3}}{z^{\widehat s_3}}.$$
Substituting these expressions into the relation
$z_2^2=z_1\bigl(z_3^2+z_1^{a_2-1}\bigr)$ gives
$$y^{2\alpha+2s_3}+x^{a_2-1}y^{2\alpha}z^{2\widehat s_3}-xz^{2\beta+2\widehat s_3}=0.$$
Hence the image of the projection is contained in $V(F)$ of $F$. Let
$\varphi\colon \mathbb C[x,y,z]\longrightarrow R_{n,q}$ be the homomorphism induced by
$$
x\longmapsto z_1,\qquad y\longmapsto z_{e-1},\qquad z\longmapsto z_e$$
Since $R_{n,q}$ is a domain, $\ker(\varphi)$ is a prime ideal. Moreover,
by Lemma \ref{elimination-D}, every generator $z_2,\ldots,z_{e-2}$ of $R_{n,q}$ is a rational function of
$z_1,z_{e-1},z_e$. Hence
$$
\operatorname{Frac}(R_{n,q})=\mathbb C(z_1,z_{e-1},z_e)=\operatorname{Frac}\bigl(\mathbb C[x,y,z]/\ker(\varphi)\bigr).
$$
Since $\textnormal{dim} R_{n,q}=2$, it follows that $\textnormal{dim} \mathbb C[x,y,z]/ \ker(\varphi)=2$. Thus $\ker(\varphi)$ is a height-one prime ideal of the UFD $\mathbb C[x,y,z]$, and hence it is principal, say $\ker(\varphi)=(G)$ for some irreducible polynomial $G$.

\noindent Since $F\in \ker(\varphi)$, the polynomial $G$ divides $F$. On the other hand, the
equation obtained above is the unique relation among the three projected coordinates up to multiplication by a unit, and therefore $F$ is associated to $G$. Hence $\ker(\varphi)=(F)$. It follows that the image closure of the projection is $V(F)$. Since the remaining generators are rational functions of $x,y,z$, the induced map is birational. \end{proof}

\begin{cor}\label{normalization-projected-D} Assume that the inclusion
$\overline R_{n,q}\hookrightarrow R_{n,q}$ is finite. Then $R_{n,q}$ is a normalization of $\overline R_{n,q}$.
\end{cor}

\begin{proof} By Theorem \ref{projected-hypersurface-D}, the two rings have the same fraction field. Since $R_{n,q}$ is an invariant ring of a finite group acting on $\mathbb C[u,v]$, it is normal. Hence, if $\overline R_{n,q}\subset R_{n,q}$ is finite, then $R_{n,q}$ is an integral, birational, normal extension of $\overline R_{n,q}$. Hence, $R_{n,q}=\overline{\overline R_{n,q}}$ the integral closure of $\overline R_{n,q}$ in its fraction field. \end{proof}

\noindent Although the preceding construction gives an explicit hypersurface equation for every dihedral quotient singularity, the exponents
$\alpha$, $\beta$, $s_3$ and $\widehat{s}_3$ become increasingly complicated as the length of the continued fraction grows. Below, we give an example which covers a neat class of examples among $D_{n,q}$ singularities. 

\begin{ex}\label{ex-Dm} Consider the family of dihedral quotient singularities whose long branch
contains a unique vertex of weight $m>2$, namely with
$$[b]=[\underbrace{2,\ldots, 2}_{k}, m, \underbrace{2,\ldots,2}_{l}]$$
where $k,l\ge0$. 

This family is particularly important because the position of the unique
vertex of weight greater than $2$ determines the number of rank-two
special CM modules. By induction on $l$ and then on $m$, we obtain
\begin{align*}
\frac{n}{q}&=\frac{(k(m-2)+m-1)l+k(m-1)+m}{(k(m-2)+1)l+k(m-1)+1},\\
\frac{n}{n-q}&=\frac{(k(m-2)+m-1)l+k(m-1)+m}{(m-2)l+m-1}.
\end{align*}

\noindent The second fraction has the continued fraction expansion 
$$[a]=[k+2,\underbrace{2,\ldots, 2}_{m-3}, l+2].$$ 
Hence, we have $e=m+1$, $\alpha=l+2$, $\beta=1$, $s_3=(m-3)l+m-2$ and $\widehat s_3=m-3$. Therefore the projected hypersurface equation is
\begin{equation}\label{equ-Dm}
y^{(2m-4)l+2m-2}+x^{k+1}y^{2l+2}z^{2m-6}-xz^{2m-4}=0. \end{equation}

In particular, the parameter $m$ determines both the complexity of the projected hypersurface and the number of rank-two special
CM modules that appear on the long branch.
\end{ex}

\begin{rem}
The rational triple points $C_{m,n}=D_nA_m$ correspond to the case
$$[b]=[\underbrace{2,\ldots,2}_{n-2},3,\underbrace{2,\ldots,2}_{m}]$$
of Example \ref{ex-Dm}.  The hypersurface equation (\ref{equ-Dm}) takes the form 
$y^{2m+4}+x^{n-1}y^{2m+2}-xz^2=0$ which coincides with the equation obtained in \cite{MAG} by a different approach.
\end{rem}

\section{Rank-one modules and $2\times2$-matrix factorizations for $\mathbb{D}_{n,q}$ singularities} We first construct three families of $2\times2$ matrix factorizations of the projected hypersurface equation.

\begin{prop}\label{mf-Dnq}
Let us consider 
$$F=y^{2\alpha+2s_3}+x^{a_2-1}y^{2\alpha}z^{2\widehat s_3}-xz^{2\beta+2\widehat s_3}
$$

\begin{enumerate}

\item[(i)] For $0\le i\le2\alpha-1$ and
$0\le j\le2\beta+2\widehat s_3-1$, define
$$
A_{i,j}^{(1)}=\begin{pmatrix}
y^{2\alpha-i} & xz^j\\
z^{2\beta+2\widehat s_3-j}&y^i\bigl(
y^{2s_3}+x^{a_2-1}z^{2\widehat s_3}\bigr)
\end{pmatrix}
$$
and
$$
B_{i,j}^{(1)}=\begin{pmatrix}
y^i\bigl(y^{2s_3}+x^{a_2-1}z^{2\widehat s_3}\bigr)&-xz^j\\
-z^{2\beta+2\widehat s_3-j}&y^{2\alpha-i}
\end{pmatrix}
$$

\item[(ii)] For $1\le j\le2\alpha+2s_3-1$, define
$$
A_j^{(2)}=\begin{pmatrix}
x & y^j\\
-y^{2\alpha+2s_3-j}&z^{2\widehat s_3}\bigl(x^{a_2-2}y^{2\alpha}-z^{2\beta}\bigr)
\end{pmatrix}
$$
and
$$
B_j^{(2)}=\begin{pmatrix}
z^{2\widehat s_3}\bigl(x^{a_2-2}y^{2\alpha}-z^{2\beta}\bigr)&-y^j\\
y^{2\alpha+2s_3-j}&x\end{pmatrix}
$$

\item[(iii)] For $0\le i\le2\widehat s_3-1$ and $1\le j\le2\alpha+2s_3-1$, define
$$
A_{i,j}^{(3)}=\begin{pmatrix}
z^{2\widehat s_3-i} & y^j\\
-y^{2\alpha+2s_3-j} & xz^i \bigl(x^{a_2-2}y^{2\alpha}-z^{2\beta}\bigr)
\end{pmatrix}
$$
and
$$
B_{i,j}^{(3)}=\begin{pmatrix}
xz^i\bigl(x^{a_2-2}y^{2\alpha}-z^{2\beta}\bigr)&-y^j\\
y^{2\alpha+2s_3-j}&z^{2\widehat s_3-i}
\end{pmatrix}$$

\end{enumerate}

\noindent Then we have $A_{i,j}^{(\ell)}B_{i,j}^{(\ell)}=B_{i,j}^{(\ell)}A_{i,j}^{(\ell)}=F\textnormal{Id}_2$ with $\ell=1,3$ and $A_j^{(2)}B_j^{(2)}=B_j^{(2)}A_j^{(2)}=F\textnormal{Id}_2$. Hence each pair above determines a $2\times 2$ matrix factorization of
the projected hypersurface equation.
\end{prop}

\subsection{Fractional ideals and module comparisons} Throughout this subsection, we identify rank-one torsion-free $\overline R_{n,q}$-modules with fractional ideals in
$Q=Q(\overline R_{n,q})$. Since $\overline R_{n,q}$ is a domain, every such module embeds into $Q$, and two fractional ideals $I,J\subset Q$ are isomorphic as $\overline R_{n,q}$-modules if and only if $J=hI$ for some $h\in Q^*$. We shall therefore freely replace fractional
ideals by equivalent ones obtained by multiplication by a nonzero element of $Q$.

\noindent To compare the modules arising from the matrix factorizations with the special CM modules on the normalization, it will be useful to have different ideal presentations of the same module.

\begin{prop}\label{alternativeideals-D} The rank-one MCM modules associated with the three families in
Proposition \ref{mf-Dnq} admit the following ideal presentations.
\begin{align*}
M_{i,j}^{(1)}&\cong (y^{2\alpha-i},z^{2\beta+2\widehat s_3-j})\\
&\cong (xz^j, y^i(y^{2s_3}+x^{a_2-1}z^{2\widehat s_3})),\\
M_j^{(2)}&\cong (x,y^{2\alpha+2s_3-j})\\
&\cong (y^j, z^{2\widehat s_3}(x^{a_2-2}y^{2\alpha}-z^{2\beta})),\\
M_{i,j}^{(3)}&\cong (z^{2\widehat s_3-i},y^{2\alpha+2s_3-j})\\
&\cong (y^j, xz^i(x^{a_2-2}y^{2\alpha}-z^{2\beta})).
\end{align*}
\end{prop}

\begin{proof} Let $$
A=\begin{pmatrix}
a&b\\
c&d
\end{pmatrix}, \qquad B=\begin{pmatrix}
d&-b\\
-c&a
\end{pmatrix}
$$
be one of the matrix factorizations in Proposition \ref{mf-Dnq}. Over $\overline R_{n,q}$ we have $AB=BA=0$. Hence the second row of
$B$ induces a surjective homomorphism
$$\operatorname{coker}(A)\longrightarrow(a,c), \qquad \overline{(u,v)}\longmapsto -cu+av$$
Similarly, the first row induces a surjection
$\operatorname{coker}(A)\longrightarrow(b,d)$. All three modules are torsion-free of rank one. The kernels of these
maps therefore we have rank zero, and hence vanish since $\operatorname{coker}(A)$ is torsion-free. Thus
$\operatorname{coker}(A)\cong(a,c)\cong(b,d)$. Applying this observation to the three families in
Proposition \ref{mf-Dnq} gives the stated presentations. 
\end{proof}

\noindent The three families of matrix factorizations constructed above are not disjoint, since different parameter values may yield
isomorphic MCM modules. Before comparing these modules with Wunram's special CM modules, we first determine the isomorphisms among the corresponding ideal presentations. The following propositions describe these identifications and determine the parameter values that give rise to new isomorphism classes.

\begin{prop}\label{isomorphism12} For every $0\leq i\leq 2\alpha-1$, there is an isomorphism
$M_{i,0}^{(1)}\cong M_{2\alpha-i}^{(2)}$. Hence, the family
$$\{M_{i,0}^{(1)}\mid 0\leq i\leq 2\alpha-1\}
$$
does not give rise to any new isomorphism classes of MCM modules. More precisely, it corresponds bijectively to the family $\{M_j^{(2)}\mid 1\leq j\leq 2\alpha\}$ under the correspondence $j=2\alpha-i$.
\end{prop}

\begin{proof}
By Proposition \ref{alternativeideals-D},
$$
M_{i,0}^{(1)}\cong (x, y^i (y^{2s_3}+x^{a_2-1}z^{2\widehat s_3})).
$$
Since $a_2\geq 2$, we have
$$x^{a_2-1}y^iz^{2\widehat s_3}=x(x^{a_2-2}y^iz^{2\widehat s_3})\in (x).$$
Therefore we have
$$M_{i,0}^{(1)}\cong (x, y^{i+2s_3}+x^{a_2-1}y^iz^{2\widehat s_3})\cong (x,y^{i+2s_3}).$$
On the other hand, Proposition \ref{alternativeideals-D} yields
$$M_{2\alpha-i}^{(2)}\cong (x, y^{2\alpha+2s_3-(2\alpha-i)}) = (x,y^{i+2s_3}).$$
Hence $M_{i,0}^{(1)}\cong M_{2\alpha-i}^{(2)}$. Finally,
$0\leq i\leq 2\alpha-1$ if and only if $1\leq 2\alpha-i\leq 2\alpha $. Therefore, the correspondence $j=2\alpha-i$ defines a bijection between the two families, showing that the modules $M_{i,0}^{(1)}$ do not give
rise to any additional isomorphism classes. \end{proof}

\begin{prop}\label{isomorphism13} Assume that $s_3\geq1$. For every $0\leq i\leq2\alpha-1$ and $2\beta\leq j\leq2\beta+2\widehat s_3-1$, 
there is an isomorphism $M_{i,j}^{(1)}\cong M_{j-2\beta,i+2s_3}^{(3)}$. Thus, the family
$$
\{M_{i,j}^{(1)} \mid 0\leq i\leq2\alpha-1,\ 2\beta\leq j\leq2\beta+2\widehat s_3-1\}$$
does not give rise to any additional isomorphism classes of maximal
CM modules. More precisely, it corresponds bijectively to the family
$$
\{M_{p,q}^{(3)}\mid 0\leq p\leq2\widehat s_3-1,\ 2s_3\leq q\leq2\alpha+2s_3-1\}
$$
under the correspondence $p=j-2\beta $ and $q=i+2s_3$.
\end{prop}

\begin{proof} Put $p=j-2\beta $ and $q=i+2s_3$. Since $2\beta\leq j\leq 2\beta+2\widehat s_3-1$ we obtain
$0\leq p\leq2\widehat s_3-1$. Similarly, $0\leq i\leq 2\alpha-1$ implies that $2s_3\leq q\leq 2\alpha+2s_3-1$. Since $s_3\geq1$, we have $q\geq1$, and therefore $(p,q)$ belongs to the parameter range of the third family. By Proposition \ref{alternativeideals-D},
$$
M_{i,j}^{(1)}\cong (y^{2\alpha-i},z^{2\beta+2\widehat s_3-j}).$$
On the other hand,
$$
M_{p,q}^{(3)}\cong (z^{2\widehat s_3-p}, y^{2\alpha+2s_3-q}).$$
Substituting $p=j-2\beta $ and $q=i+2s_3$ gives
$$
M_{p,q}^{(3)}\cong (z^{2\widehat s_3-(j-2\beta)}, y^{2\alpha+2s_3-(i+2s_3)})=(z^{2\beta+2\widehat s_3-j}, y^{2\alpha-i}).
$$
Since the order of the generators of an ideal is irrelevant, 
$$
(z^{2\beta+2\widehat s_3-j}, y^{2\alpha-i})=(y^{2\alpha-i}, z^{2\beta+2\widehat s_3-j}).$$
Therefore,
$$
M_{i,j}^{(1)}\cong M_{j-2\beta,i+2s_3}^{(3)}$$
The correspondence $(p,q)=(j-2\beta, i+2s_3)$ is clearly bijective between the two parameter sets. \end{proof}

\begin{rem}\label{family23-D} Propositions \ref{isomorphism12} and \ref{isomorphism13} show that
certain members of the first family are already represented in the second and third families. Thus, these subfamilies do not
contribute any new isomorphism classes of rank-one MCM modules. On the other hand, the same argument does not lead to a corresponding identification between the second and third families. Indeed, by
Proposition \ref{alternativeideals-D},
$$
M_j^{(2)}\cong (x,y^{2\alpha+2s_3-j})\cong (y^j, z^{2\widehat s_3}(x^{a_2-2}y^{2\alpha}-z^{2\beta}),)$$
$$
M_{i,k}^{(3)}\cong (z^{2\widehat s_3-i},y^{2\alpha+2s_3-k})\cong (y^k, xz^i(x^{a_2-2}y^{2\alpha}-z^{2\beta})).$$
Unlike the previous cases, these ideal presentations do not immediately suggest a parameter correspondence that would identify the two modules. Therefore, the preceding propositions do not rule out the possibility of
additional isomorphisms, either within a single family or between the
second and third families.
\end{rem}

\noindent We now compare the rank-one MCM modules arising from the matrix factorizations of the projected hypersurface with Wunram's special CM modules over the normalization $R_{n,q}$.

\begin{thm} \cite{Wunram} Let $\pi\colon\widetilde X\longrightarrow \operatorname{Spec}(R_{n,q})$
be the minimal resolution. Then there is a one-to-one correspondence
between the irreducible components $E_1,\ldots, E_r$ of the exceptional divisor and the isomorphism classes of non-free
indecomposable special CM $R_{n,q}$-modules. Moreover, if $M_i$ is the special Cohen--Macaulay module corresponding to $E_i$ and $\widetilde M_i$ denotes the associated full sheaf, then
$$c_1(\widetilde M_i)\cdot E_j=\delta_{ij}$$
Furthermore, the rank of $M_i$ is equal to the coefficient of $E_i$ in the fundamental cycle.
\end{thm}

\noindent Thus the rank-one special CM modules correspond precisely to the exceptional curves having coefficient $1$ in the fundamental cycle. We determine which of these modules arise as reflexive normalizations of the rank-one modules constructed above. For this purpose, we use Wemyss's explicit description of the rank-one special CM modules as two-generated fractional ideals of $R_{n,q}$. Assume that 
$ \displaystyle \frac{n}{q}=[b_1,\ldots,b_r] $ and let $\nu$ be the largest integer such that $b_1=\cdots=b_\nu=2$, with $\nu=0$ if $b_1\geq 3$. The associated $i$-series is defined by $$ i_0=n,\qquad i_1=q,\qquad i_t=b_{t-1}i_{t-1}-i_{t-2}, \quad 2\leq t\leq r. $$

\begin{rem}
Wunram's geometric description can also be expressed in terms of the
dual cycles of the exceptional divisor. Let
$I=((E_i\cdot E_j))$ be the intersection matrix of the minimal
resolution and define the dual cycles by
$$
E_i^*\cdot E_j=-\delta_{ij}.
$$
By \cite{RomanoNemethi}, we get $E_i^*=\sum_j(-I^{-1})_{ij}E_j$. Thus the dual cycles corresponding to the rank-one special CM modules
can be read directly from the rows of $-I^{-1}$. 
\end{rem}

\noindent Following \cite{WemyssDII}, let
$$l_t=\begin{cases}2, & t=\nu+1,\\
2+\displaystyle\sum_{p=\nu+1}^{t-1}(b_p-2), & \nu+2\leq t\leq r
\end{cases}
$$
Thus, $l_t$ records the position in the dual continued fraction of the coefficient corresponding to the exceptional curve $E_t$ on the long branch of the resolution graph. 

\begin{thm} \cite{WemyssDII} \label{wemysd2}
Define 
$$
w_1=uv,\quad w_2=(u^q+v^q)(u^q+(-1)^{a_2}v^q),\quad w_3=(u^q-v^q)(u^q+(-1)^{a_2}v^q)$$
For $\nu+1\leq t\leq r$ with $r\geq 1$, define
$$
\Delta_t=1+\sum_{j=\nu+1}^{t-1}c_{l_j},\qquad \Gamma_t=\sum_{j=\nu+1}^{t-1}d_{l_j}$$
where an empty sum is understood to be zero. Then the non-free indecomposable rank-one special CM
$R_{n,q}$-modules are $W_+$, $W_-$, $W_{i_{\nu+1}}$, $\ldots $, $W_{i_r}$ where
$$W_+=R_{n,q}(u^q+v^q)+R_{n,q}w_1^{n-q}(u^q-v^q),$$
$$W_-=R_{n,q}(u^q-v^q)+R_{n,q}w_1^{n-q}(u^q+v^q),$$
$$W_{i_t}=R_{n,q}w_1^{i_t}+R_{n,q}w_2^{\Delta_t}w_3^{\Gamma_t}, \qquad \nu+1\leq t\leq r.$$
Moreover, each of these modules is minimally generated by two elements,
and every non-free indecomposable rank-one special CM
$R_{n,q}$-module is isomorphic to one of them.
\end{thm}

\noindent We begin with the endpoint module $W_1$. The following proposition shows that $W_1$ is recovered as the reflexive
normalization of a distinguished member of the third matrix factorization family.

\begin{prop} \label{W1-MF0} The endpoint rank-one special CM module satisfies
$W_{i_r}=W_1\cong(Z_{e-1}, Z_e)\cong (z,y)$. Moreover,
$$\nu^*M_{2\widehat s_3-1,\,2\alpha+2s_3-1}^{(3)}\cong W_1.$$
\end{prop}


\begin{proof} By Theorem \ref{wemysd2}, we have  $W_1=W_{i_r}=(w_1,w_2^{\Delta_r}w_3^{\Gamma_r})$. 
Since $l_r=e-1$, $i_r=1$ and $i_{r+1}=0$, we have $r_{e-1}=1$. Moreover, $r_e=0$, $c_e-c_{e-1}=\Delta_r $ and $d_e-d_{e-1}=\Gamma_r$. It follows that
$$
\frac{Z_e}{Z_{e-1}}=w_1^{-1}w_2^{\Delta_r}w_3^{\Gamma_r}$$
and hence
$$
w_2^{\Delta_r}w_3^{\Gamma_r}=w_1\frac{Z_e}{Z_{e-1}}.
$$
Multiplying $W_1$ by $Z_{e-1}$ gives  $Z_{e-1}W_1=w_1(Z_{e-1},Z_e)$. So, we get 
$W_1\cong(Z_{e-1}, Z_e)\cong (y, z)$. On the other hand, Proposition \ref{alternativeideals-D} gives
$$M_{i,j}^{(3)}\cong (z^{2\widehat s_3-i},y^{2\alpha+2s_3-j})$$
Taking $i=2\widehat s_3-1$ and $j=2\alpha+2s_3-1$, we obtain
$$M_{2\widehat s_3-1,2\alpha+2s_3-1}^{(3)}\cong(z,y)$$
Thus, the matrix factorization
$$(A_{2\widehat s_3-1,2\alpha+2s_3-1}^{(3)},B_{2\widehat s_3-1,2\alpha+2s_3-1}^{(3)})
$$
recovers $W_1$ after reflexive normalization. Consequently, we have
$$\nu^*M_{2\widehat s_3-1,2\alpha+2s_3-1}^{(3)}\cong W_1.$$
\end{proof}

\noindent We next identify the matrix factorization corresponding to the terminal special module $W_+$. We first express $W_+$ in terms of the invariants $v_1,v_2,v_3$.

\begin{lem}\label{Wplus-presentation} The special CM module
$$W_+=R_{n,q}(u^q+v^q)+R_{n,q}(uv)^{n-q}(u^q-v^q)
$$
admits the fractional-ideal presentation $W_+\cong(v_2,v_1^{n-q}v_3)$.
\end{lem}

\begin{proof} Set $\eta=u^q+(-1)^{a_2}v^q$. Multiplying $W_+$ by $\eta$ gives
$$\eta W_+=\bigl(\eta(u^q+v^q), \eta(uv)^{n-q}(u^q-v^q)\bigr)=(v_2,v_1^{n-q}v_3)$$
Since $\eta$ is nonzero, multiplication by $\eta$ preserves the isomorphism class of the fractional ideal, and hence
$W_+\cong(v_2,v_1^{n-q}v_3)$. \end{proof}

\noindent We now identify the matrix factorization corresponding to $W_+$.

\begin{prop}\label{Wplus-D} Assume that $\alpha\geq s_3$. Then $\nu^*M_{\alpha-s_3,\,\beta+\widehat s_3}^{(1)}\cong W_+$.
\end{prop}

\begin{proof} By Lemma \ref{Wplus-presentation},
$W_+\cong(v_2,v_1^{n-q}v_3)$. Since
$$
Z_2=v_1^{r_2}v_2,\qquad Z_3=v_1^{r_3}v_3,\qquad r_2-r_3=n-q$$
multiplication by $v_1^{r_2}$ gives
$$v_1^{r_2}W_+\cong (v_1^{r_2}v_2,v_1^{r_2+n-q}v_3)=(Z_2,v_1^{2r_2-r_3}v_3)$$
Since $2r_2-r_3=2(n-q)+r_3$ and $Z_1=v_1^{2(n-q)}$, we obtain
$W_+\cong(Z_2,Z_1Z_3)$. Under the projection $x=Z_1$, $y=Z_{e-1}$ and $z=Z_e$, Lemma \ref{elimination-D} gives
$$
Z_2=\frac{xz^\beta}{y^\alpha},\qquad Z_3=\frac{y^{s_3}}{z^{\widehat s_3}}$$
Hence
$$W_+\cong \left(\frac{xz^\beta}{y^\alpha}, \frac{xy^{s_3}}{z^{\widehat s_3}}\right)\cong \left(y^{\alpha+s_3},z^{\beta+\widehat s_3}\right)$$
On the other hand, Proposition \ref{alternativeideals-D} gives
$M_{i,j}^{(1)}\cong (y^{2\alpha-i}, z^{2\beta+2\widehat s_3-j})$. Taking $i=\alpha-s_3$ and $j=\beta+\widehat s_3$, we obtain
$$M_{\alpha-s_3,\,\beta+\widehat s_3}^{(1)}\cong \left(y^{\alpha+s_3},z^{\beta+\widehat s_3}\right)\cong W_+.$$
Thus $W_+$ is recovered from the matrix factorization
$$(A_{\alpha-s_3,\,\beta+\widehat s_3}^{(1)}, B_{\alpha-s_3,\,\beta+\widehat s_3}^{(1)})$$
and after reflexive normalization, we get 
$\nu^*M_{\alpha-s_3,\,\beta+\widehat s_3}^{(1)}\cong W_+$. \end{proof}

\noindent We next consider the second terminal special module $W_-$. Its projected ideal presentation is obtained similarly.

\begin{prop}\label{Wminus-projected} The rank-one special CM module $W_-$ admits the projected
fractional-ideal presentation
$$W_-\cong (y^{\alpha+s_3},xz^{\beta+\widehat s_3}).$$
\end{prop}
 
\begin{proof} Let $\eta=u^q+(-1)^{a_2}v^q$. Multiplying $W_-$ by $\eta$ gives
$\eta W_-=(v_3,v_1^{n-q}v_2)$ and hence $W_-\cong(v_3,v_1^{n-q}v_2)$. Since
$$
Z_2=v_1^{r_2}v_2,\qquad Z_3=v_1^{r_3}v_3,\qquad r_2-r_3=n-q$$
multiplication by $v_1^{r_3}$ gives
$$v_1^{r_3}W_- \cong (v_1^{r_3}v_3,v_1^{r_3+n-q}v_2)=(Z_3,Z_2).$$
Thus $W_-\cong(Z_3,Z_2)$. Using
$$Z_2=\frac{xz^\beta}{y^\alpha}, \qquad Z_3=\frac{y^{s_3}}{z^{\widehat s_3}}$$
and multiplying by $y^\alpha z^{\widehat s_3}$, we obtain $W_-\cong (y^{\alpha+s_3},xz^{\beta+\widehat s_3})$. \end{proof}

\begin{cor}\label{Wpm-symmetry}
Put $d=\alpha+s_3$ and $h=\beta+\widehat s_3$. Then the two terminal special CM modules $W_+$ and $W_-$ are recovered
from the two cokernels of the same matrix factorization $(A^{(1)}_{\alpha-s_3, h}, B^{(1)}_{\alpha-s_3, h})$. More precisely,
$$
\nu^*\operatorname{coker}(A^{(1)}_{\alpha-s_3,h})\cong W_+, \qquad \nu^*\operatorname{coker} (B^{(1)}_{\alpha-s_3, h})\cong W_-.$$
\end{cor}

\begin{proof} By Proposition \ref{Wplus-D}, we have $\operatorname{coker}(A^{(1)}_{\alpha-s_3,h})\cong (y^d, z^h)$
whose reflexive normalization is $W_+$. On the other hand, $\operatorname{coker} (B^{(1)}_{\alpha-s_3,h})\cong (y^d, xz^h)$. 
By Proposition \ref{Wminus-projected}, we obtain that $W_-\cong(y^d,xz^h)$ and the result follows after reflexive normalization. \end{proof}

\subsection{Reflexive normalization of the matrix-factorization modules} We now compare Wemyss's rank-one special CM modules with the reflexive normalizations of the rank-one modules arising from the three families of matrix factorizations.

\begin{prop}\label{normalization-MF-D} For a rank-one MCM $\overline R_{n,q}$-module $M$, let
$\nu^*M=(M\otimes_{\overline R_{n,q}}R_{n,q})^{\vee\vee}$. Then
\begin{align*}
\nu^*M_{i,j}^{(1)}&\cong \left(Y^{2\alpha-i}, Z^{2\beta+2\widehat s_3-j}\right)^{\vee\vee},\\
\nu^*M_j^{(2)}&\cong \left(X,Y^{2\alpha+2s_3-j}\right)^{\vee\vee},\\
\nu^*M_{i,j}^{(3)}&\cong \left(Z^{2\widehat s_3-i}, Y^{2\alpha+2s_3-j}\right)^{\vee\vee}.
\end{align*}
\end{prop}

\begin{proof} By Proposition \ref{alternativeideals-D},
\begin{align*}
M_{i,j}^{(1)}&\cong \left(y^{2\alpha-i}, z^{2\beta+2\widehat s_3-j}\right),\\
M_j^{(2)}&\cong \left(x,y^{2\alpha+2s_3-j}\right),\\
M_{i,j}^{(3)}&\cong \left(z^{2\widehat s_3-i}, y^{2\alpha+2s_3-j}\right)
\end{align*}

\noindent Let $I$ denote any one of these ideals. Extension of scalars along $\iota$ gives a natural surjection
$$I\otimes_{\overline R_{n,q}}R_{n,q}\longrightarrow IR_{n,q},\qquad a\otimes r\longmapsto\iota(a)r$$
Since $\overline R_{n,q}$ and $R_{n,q}$ have the same fraction field, this map becomes an isomorphism over the fraction field. Its kernel is
therefore torsion, and taking reflexive hulls gives
$$
\left(I\otimes_{\overline R_{n,q}}R_{n,q}\right)^{\vee\vee}\cong \left(IR_{n,q}\right)^{\vee\vee}$$
The result now follows from $x\mapsto X$, $y\mapsto Y$ and $z\mapsto Z$. \end{proof}

\subsubsection{The terminal modules $W_+$ and $W_-$} Consider the two rank-one special CM modules corresponding to
the short branches. Put $\eta=u^q+(-1)^{a_2}v^q$. The two modules admit parallel fractional ideal descriptions which
allow us to express them in the projected coordinates.

\begin{prop}\label{Wpm-projected} The terminal special CM modules satisfy
$$W_+\cong \left(y,z^{\alpha+s_3}\right), \qquad W_-\cong \left(xy,z^{\alpha+s_3}\right)$$
Moreover, if $2\beta+2\widehat s_3\geq\alpha+s_3$ then
$$
\nu^*M_{2\alpha-1, 2\beta+2\widehat s_3-\alpha-s_3}^{(1)}\cong W_+.$$
\end{prop}

\begin{proof} Multiplying the generators of $W_+$ by $\eta$ gives
$$\eta W_+=\bigl(\eta(u^q+v^q), \eta(uv)^{n-q}(u^q-v^q)\bigr)=(v_2,v_1^{n-q}v_3)$$
Hence $W_+\cong(v_2,v_1^{n-q}v_3)$. Since
$$Z_2=v_1^{r_2}v_2,\qquad Z_3=v_1^{r_3}v_3,\qquad r_2-r_3=n-q$$
multiplication by $v_1^{r_2}$ gives $W_+\cong (Z_2, Z_1Z_3)$ where we have used $Z_1=v_1^{2(n-q)}$. Substituting
$$Z_2=\frac{xy^{1-\widehat s_3}}{z^\alpha}, \qquad Z_3=\frac{z^{s_3}}{y^{\widehat s_3}}$$
and multiplying by $\displaystyle \frac{z^\alpha y^{\widehat s_3}}{x}$, we obtain $W_+\cong(y,z^{\alpha+s_3})$.

\noindent The computation for $W_-$ is analogous. Multiplication by $\eta$ gives $W_-\cong (v_3, v_1^{n-q}v_2)$. 
Multiplying by $v_1^{r_3}$ and using $r_2-r_3=n-q$, we obtain $W_-\cong (Z_3, Z_2)$. The same substitutions now give
$$W_-\cong \left(\frac{z^{s_3}}{y^{\widehat s_3}}, \frac{xy^{1-\widehat s_3}}{z^\alpha}\right)\cong (z^{\alpha+s_3}, xy)$$
So, we have $W_-\cong (xy, z^{\alpha+s_3})$. It remains to identify $W_+$ with one of the matrix factorization
modules. By Proposition \ref{alternativeideals-D}, 
$$M_{i,j}^{(1)}\cong \left(y^{2\alpha-i}, z^{2\beta+2\widehat s_3-j}\right).$$
Taking $i=2\alpha-1$ and $j=2\beta+2\widehat s_3-\alpha-s_3$ we obtain
$$M_{2\alpha-1, 2\beta+2\widehat s_3-\alpha-s_3}^{(1)}\cong \left(y,z^{\alpha+s_3}\right)\cong W_+.$$
After reflexive normalization, we obtain $\nu^*M_{2\alpha-1, 2\beta+2\widehat s_3-\alpha-s_3}^{(1)}\cong W_+$.  \end{proof}

\subsubsection{The special modules along the long branch} We now turn to the special CM modules corresponding to the vertices of
the long branch. Their ranks depend on the coefficient $b_3$. If $b_3\geq 3$, every exceptional curve on the long branch occurs with
coefficient $1$ in the fundamental cycle, and hence all the corresponding
special CM modules have rank one. We treat this case first. When $b_3=2$, some of the initial vertices occur with coefficient $2$ and
give rise to rank-two special CM modules; this case will be considered separately.

\noindent Assume that $b_3\geq3$. Then $\nu=0$, and by Theorem \ref{wemysd2} the special CM modules corresponding to the long branch
are
$$W_{i_t}=\left(v_1^{i_t},v_2^{\Delta_t}v_3^{\Gamma_t}\right),\qquad 1\leq t\leq r$$
The endpoint module $W_{i_r}=W_1$ has already been identified in Proposition \ref{W1-MF0}. We now determine the projected presentations
of the remaining modules and compare them with the matrix factorization modules.

\begin{prop}\label{W-intermediate} Assume that $b_3\geq 3$. Then, for every $1\leq t\leq r$, the
rank-one special CM module $W_{i_t}$ admits the fractional-ideal presentation
$$W_{i_t}\cong \left(v_1^q, v_2\prod_{p=1}^{t-1}Z_{l_p}\right)$$
where an empty product is $1$.
\end{prop}

\begin{proof} Since $b_3\geq 3$, we have $\nu=0$ and $i_1=q$. By Theorem \ref{wemysd2},
$$
\Delta_t=1+\sum_{p=1}^{t-1}c_{l_p},\qquad \Gamma_t=\sum_{p=1}^{t-1}d_{l_p}$$
Moreover we have $r_{l_p}=i_p-i_{p+1}$ and therefore 
$$
\sum_{p=1}^{t-1}r_{l_p}=\sum_{p=1}^{t-1}(i_p-i_{p+1})=q-i_t.$$
Using $Z_{l_p}=v_1^{r_{l_p}}v_2^{c_{l_p}}v_3^{d_{l_p}}$ we obtain
$$v_2\prod_{p=1}^{t-1}Z_{l_p}=v_1^{\sum_{p=1}^{t-1}r_{l_p}}v_2^{1+\sum_{p=1}^{t-1}c_{l_p}}v_3^{\sum_{p=1}^{t-1}d_{l_p}}=
v_1^{q-i_t}v_2^{\Delta_t}v_3^{\Gamma_t}$$
Hence we get 
$$v_1^{q-i_t}W_{i_t}=\left(v_1^q, v_1^{q-i_t}v_2^{\Delta_t}v_3^{\Gamma_t}\right)=\left(v_1^q, v_2\prod_{p=1}^{t-1}Z_{l_p}\right)$$
Since multiplication by the nonzero rational function $v_1^{q-i_t}$ does not change the isomorphism class of a fractional
ideal, the result follows.  \end{proof}

\noindent We next express these modules in the projected coordinates. For $t=1$, Proposition \ref{W-intermediate} gives
$W_{i_1}\cong(v_1^q,v_2)$. Since
$$Z_1=v_1^{2(n-q)},\qquad Z_2=v_1^{2(n-q)-q}v_2,\qquad Z_2=\frac{xz^\beta}{y^\alpha}$$
we obtain $v_2=v_1^q\frac{z^\beta}{y^\alpha}$. Multiplying by $y^\alpha/v_1^q$ we obtain
$W_{i_1}\cong(y^\alpha,z^\beta)$. For $2\leq t\leq r$, put
$$S_t=\sum_{p=2}^{t-1}s_{l_p},\qquad \widehat S_t=\sum_{p=2}^{t-1}\widehat s_{l_p}$$
where an empty sum is  zero.

\noindent Since $l_1=2$, we have 
$$v_2\prod_{p=1}^{t-1}Z_{l_p}=v_2Z_2\prod_{p=2}^{t-1}Z_{l_p}.$$
For $p\geq2$, the elimination formulas give
$$Z_{l_p}=\frac{y^{s_{l_p}}}{z^{\widehat s_{l_p}}}$$
and hence
$$\prod_{p=2}^{t-1}Z_{l_p}=\frac{y^{S_t}}{z^{\widehat S_t}}.$$
Using $v_2=v_1^q\frac{z^\beta}{y^\alpha}$ and $Z_2=\frac{xz^\beta}{y^\alpha}$
we obtain
$$v_2\prod_{p=1}^{t-1}Z_{l_p}=v_1^qxy^{S_t-2\alpha}z^{2\beta-\widehat S_t}.$$
Therefore Proposition \ref{W-intermediate} gives
$$W_{i_t} \cong \left(y^{2\alpha-S_t}, xz^{2\beta-\widehat S_t}\right),\qquad 2\leq t\leq r.$$

\noindent These projected presentations can now be compared directly with the matrix factorizations. We begin with $W_{i_1}$. By Proposition \ref{alternativeideals-D},
$$M_{i,j}^{(1)}\cong \left(y^{2\alpha-i}, z^{2\beta+2\widehat s_3-j}\right).
$$
Since $W_{i_1}\cong \left(y^\alpha,z^\beta\right)$, taking $i=\alpha $ and $j=\beta+2\widehat s_3$ we get 
$$\nu^*M_{\alpha,\beta+2\widehat s_3}^{(1)}\cong \left(Y^\alpha, Z^\beta\right)^{\vee\vee}\cong W_{i_1}.$$

\noindent It remains to consider $W_{i_t}$ for $2\leq t\leq r-1$, since the endpoint $W_{i_r}=W_1$ has already been treated in Proposition \ref{W1-MF0}. For these modules,
$$W_{i_t}\cong \left(y^{2\alpha-S_t}, xz^{2\beta-\widehat S_t}\right).$$
On the other hand, the cokernel of the second matrix in the first family satisfies
$$\operatorname{coker}B_{i,j}^{(1)}\cong \left(xz^j, y^{2\alpha-i}\right)$$
Thus the natural choice of indices is $i=S_t$ and $j=2\beta-\widehat S_t$. To verify that these indices lie in the parameter range of the first family, we record the following bounds.

\begin{lem}\label{bounds-St} Assume that $b_3\geq3$. For $2\leq t\leq r-1$, one has
$0\leq S_t\leq\alpha-1$ and $0\leq\widehat S_t\leq\beta-1$. Hence, we have
$0\leq S_t\leq2\alpha-1$ and $0\leq 2\beta-\widehat S_t \leq 2\beta+2\widehat s_3-1$. Thus
$(S_t, 2\beta-\widehat S_t)$ lies in the parameter range of the first family of matrix factorizations.
\end{lem}

\begin{proof} By Riemenschneider's point rule, for every $3\leq j\leq e-1$, the number of indices $p\in\{2,\ldots,r\}$ such that $l_p=j$
is $a_j-2$. Hence
$$
\sum_{p=2}^r s_{l_p}=\sum_{j=3}^{e-1}(a_j-2)s_j=\alpha-1.$$
Similarly,
$$
\sum_{p=2}^r\widehat s_{l_p}=\sum_{j=3}^{e-1}(a_j-2)\widehat s_j=\sum_{j=3}^{e-2}(a_j-2)\widehat s_j=\beta-1$$
where we have used $\widehat s_{e-1}=0$. Since
$$S_t=\sum_{p=2}^{t-1}s_{l_p},\qquad \widehat S_t=\sum_{p=2}^{t-1}\widehat s_{l_p}$$
and all the terms in these sums are nonnegative, they are partial sums of the preceding total sums. Therefore
have $0\leq S_t\leq\alpha-1$ and $0\leq\widehat S_t\leq\beta-1$. It follows immediately that
$0\leq S_t\leq\alpha-1\leq2\alpha-1$. Moreover,
$2\beta-\widehat S_t\leq2\beta $. Since $\widehat s_3\geq1$, we have
$2\beta\leq2\beta+2\widehat s_3-1$. On the other hand,
$$2\beta-\widehat S_t\geq2\beta-(\beta-1)=\beta+1>0$$
Thus we obtain $0\leq2\beta-\widehat S_t\leq2\beta+2\widehat s_3-1$. Hence
$(S_t,2\beta-\widehat S_t)$ lies in the parameter range of the first family. \end{proof}

\begin{prop}\label{Wit-MF} Assume that $b_3\geq3$. For every $2\leq t\leq r-1$,
$$\nu^*\operatorname{coker}B^{(1)}_{S_t,\,2\beta-\widehat S_t}\cong W_{i_t}.$$
\end{prop}

\begin{proof} By the projected presentation obtained above,
$$W_{i_t}\cong \left(Y^{2\alpha-S_t}, XZ^{2\beta-\widehat S_t}\right)$$
as a fractional ideal of $R_{n,q}$. On the other hand,
$$\operatorname{coker}B_{i,j}^{(1)}\cong \left(xz^j, y^{2\alpha-i}\right)$$
as a fractional ideal of $\overline R_{n,q}$. Taking $i=S_t$ and $j=2\beta-\widehat S_t$ we obtain
$$\operatorname{coker}B^{(1)}_{S_t,\,2\beta-\widehat S_t}\cong \left(xz^{2\beta-\widehat S_t}, y^{2\alpha-S_t}\right).$$
By Lemma \ref{bounds-St}, these indices lie in the parameter range of the first family. Extending to $R_{n,q}$ and taking the reflexive hull gives
$$\nu^*\operatorname{coker}B^{(1)}_{S_t,\,2\beta-\widehat S_t}\cong \left(XZ^{2\beta-\widehat S_t}, Y^{2\alpha-S_t}\right)^{\vee\vee}.$$
By the projected presentation of $W_{i_t}$ obtained above, we have 
$$W_{i_t}\cong \left(Y^{2\alpha-S_t}, XZ^{2\beta-\widehat S_t}\right)^{\vee\vee}$$
and the result follows. \end{proof}

\noindent We now assume that $b_3=2$. Let $\nu$ be the largest integer such
that $b_3=b_4=\cdots=b_\nu=2$. Then the exceptional curves $E_3,\ldots,E_\nu$ occur with coefficient
$2$ in the fundamental cycle and therefore correspond, by Theorem \ref{thm-wunram}, to rank-two special CM modules. These modules cannot arise from the rank-one modules associated with the $2\times2$ matrix factorizations considered above. We therefore construct $4\times4$
matrix factorizations whose cokernels have rank two and compare their reflexive normalizations with the special CM modules corresponding to $E_3,\ldots,E_\nu$.

\section{Rank-two modules and $4\times4$-matrix factorizations for $\mathbb{D}_{n,q}$-singularities}\label{sect-dnq2}

Recall the defining equation $F$ of the projected hypersurface given by Theorem \ref{projected-hypersurface-D}. To simplify the notation, let us write
$$
F=y^N+x^{a_2-1}y^{2\alpha}z^H-xz^L
$$
where $N=2\alpha+2s_3$, $H=2\widehat s_3$ $L=2\beta+2\widehat s_3$. 

We first construct a family of $4\times4$ matrix factorizations by splitting each of the three monomials of $F$ into two monomial
factors. For integers
$$
1\leq r\leq N-1,\qquad 0\leq u\leq a_2-1,\qquad 0\leq v\leq2\alpha,\qquad 0\leq w\leq H,\qquad 1\leq\ell\leq L-1$$
define $\theta:=(r,u,v,w,\delta,\ell)$ with $\delta\in{0,1}$ and let
$$p_1=y^r,\qquad q_1=y^{N-r}, \quad p_2=x^uy^vz^w,\quad q_2=x^{a_2-1-u}y^{2\alpha-v}z^{H-w}$$
$$p_3=(-1)^\delta x^\delta z^\ell,\qquad q_3=(-1)^{1-\delta}x^{1-\delta}z^{L-\ell}.$$
Then we have  $p_1q_1+p_2q_2+p_3q_3=F$. Moreover, define 
$$
\Phi_\theta :=\begin{pmatrix}
p_1&-q_2&-q_3&0\\
p_2&q_1&0&-q_3\\
p_3&0&q_1&q_2\\
0&p_3&-p_2&p_1
\end{pmatrix}, \qquad \Psi_\theta :=\begin{pmatrix}
q_1&q_2&q_3&0\\
-p_2&p_1&0&q_3\\
-p_3&0&p_1&-q_2\\
0&-p_3&p_2&q_1
\end{pmatrix}.
$$

\begin{prop}\label{monomial-koszul-mf} For every admissible parameter $\theta$,
$\Phi_\theta\Psi_\theta =\Psi_\theta\Phi_\theta =F\textnormal{Id}_4$.  Hence $(\Phi_\theta,\Psi_\theta)$ is a $4\times4$ matrix
factorization of $F$.
\end{prop}

\begin{proof}
This is the standard three-term Koszul matrix factorization associated to
$F=p_1q_1+p_2q_2+p_3q_3$. A direct multiplication gives
$\Phi_\theta\Psi_\theta =\Psi_\theta\Phi_\theta =(p_1q_1+p_2q_2+p_3q_3)\textnormal{Id}_4=F\textnormal{Id}_4$. 
\end{proof}

 Let $E_\theta=\operatorname{coker}(\Phi_\theta)$. Since $F$ is a non-zero-divisor in the regular ring $S$,
$E_\theta$ is an MCM $A$-module. We shall mainly use the minimal members of this family. Let
$\Theta_{\min}$ denote the set of parameters $\theta$ for which all six monomials
$p_1, q_1, p_2,q_2, p_3, q_3$ belong to the maximal ideal
$\mathfrak m=(x,y,z)\subseteq A$. Equivalently, we have $1\leq r\leq N-1$ and $1\leq\ell\leq L-1$, and neither $p_2$ nor $q_2$ is a unit.

\begin{prop}\label{minimal-koszul-rank-two}
Let $\theta\in\Theta_{\min}$. Then $E_\theta$ is a rank-two MCM $A$-module minimally generated by four elements.
\end{prop}

\begin{proof}
Since every entry of $\Phi_\theta$ and $\Psi_\theta$ belongs to $\mathfrak m$, the corresponding $2$-periodic resolution is minimal.
Hence $\mu_A(E_\theta)=4$. It remains to compute the rank.  Moreover,
$\operatorname{adj}(\Phi_\theta)=F\Psi_\theta $. Thus every $3\times3$ minor of $\Phi_\theta$ is divisible by $F$, and
hence every $3\times3$ minor of the reduction $\overline\Phi_\theta$ vanishes in $A$. Therefore
$\operatorname{rank}_K(\overline\Phi_\theta)\leq 2$ where $K=\operatorname{Frac}(A)$. On the other hand, the minor obtained from the first two rows and the first two columns is
$$\det\begin{pmatrix}
p_1&-q_2\\
p_2&q_1
\end{pmatrix}=p_1q_1+p_2q_2=-p_3q_3$$
in $A$. Since $A$ is a domain and $p_3q_3\neq0$, this minor is
nonzero. Hence $\operatorname{rank}_K(\overline\Phi_\theta)=2$. Hence we obtain $\operatorname{rank}_A(E_\theta)=4-2=2$. \end{proof}

\noindent We next compare these modules with modules over the normalization. Let
$\iota\colon A\hookrightarrow B$ be the normalization map, and write
$P_i=\iota(p_i), Q_i=\iota(q_i)$. Then we have $P_1Q_1+P_2Q_2+P_3Q_3=0$ in $B$. Let $\Phi_{\theta,B}$ and $\Psi_{\theta,B}$ denote the matrices obtained from $\Phi_\theta$ and $\Psi_\theta$ by applying
$\iota$ entrywise. Consider the first two rows of $\Psi_{\theta,B}$:
$$
T_{\theta,B}=\begin{pmatrix}
Q_1&Q_2&Q_3&0\\
-P_2&P_1&0&Q_3
\end{pmatrix},
$$
and put $L_{\theta,B}=\operatorname{im}(T_{\theta,B})\subseteq B^2$. Equivalently,
$$L_{\theta,B}=B\binom{Q_1}{-P_2}+B\binom{Q_2}{P_1}+B\binom{Q_3}{0}+B\binom{0}{Q_3}.$$

\begin{prop}\label{lattice-general-koszul}
For every $\theta\in\Theta_{min}$, we have 
$$(E_\theta\otimes_A B)^{\vee\vee}\cong L_{\theta,B}^{\vee\vee}.$$
\end{prop}

\begin{proof} Since $\Psi_{\theta,B}\Phi_{\theta,B}=0$ the matrix $T_{\theta,B}$ satisfies
$T_{\theta,B}\Phi_{\theta,B}=0$. Hence it induces a surjective homomorphism
$$
\eta_\theta: \operatorname{coker}(\Phi_{\theta,B})\longrightarrow L_{\theta,B}.$$
The module $\operatorname{coker}(\Phi_{\theta,B})\cong E_\theta\otimes_A B$ has rank two. The matrix $T_{\theta,B}$ also has generic rank two. Indeed, the determinant of its first two columns is $P_1Q_1+P_2Q_2=-P_3Q_3$ which is nonzero in the domain $B$. Thus $L_{\theta,B}$ has rank two. It follows that $\eta_\theta$ becomes an isomorphism after tensoring
with $\operatorname{Frac}(B)$. Therefore its kernel is torsion. Conversely, since $L_{\theta,B}\subseteq B^2$ is torsion-free, every
torsion element of $\operatorname{coker}(\Phi_{\theta,B})$ belongs to
$\ker(\eta_\theta)$. Hence
$$\ker(\eta_\theta)=\operatorname{tor}\bigl(\operatorname{coker}(\Phi_{\theta,B})\bigr)$$
and we get 
$$(\operatorname{coker}(\Phi_{\theta,B}))^{\vee\vee}\cong L_{\theta,B}^{\vee\vee}$$
Since $\operatorname{coker}(\Phi_{\theta,B})\cong E_\theta\otimes_A B$ the result follows. \end{proof}

\begin{cor}\label{koszul-special-test} Let $U$ be an indecomposable rank-two special CM $B$-module. If, for
some $\theta\in\Theta_{min}$, $L_{\theta,B}^{\vee\vee}\cong U$ then
$$(E_\theta\otimes_A B)^{\vee\vee}\cong U.$$
\end{cor}

\noindent In the following, we study the rank-two special modules of the dihedral quotient singularity. Assume that $b_3=2$, and let $\nu$ be the largest integer such that $b_3=b_4=\cdots=b_{\nu+2}=2$. Then there are precisely $\nu$ rank-two special CM modules. We denote them by $U_1,\ldots, U_\nu $. In the notation of Iyama--Wemyss, they are
$$U_s=V_{i_{\nu+1}+(\nu-s)(n-q)}, \qquad 1\leq s\leq\nu $$
where $V_t=(\mathbb C[u,v]\otimes\rho_t)^{\mathbb D_{n,q}}$. Thus
$U_\nu=V_{i_{\nu+1}}$ and the rank-two modules occur consecutively along the coefficient-$2$
part of the long branch.

\noindent The preceding construction reduces the comparison with the projected hypersurface to an explicit lattice problem. Namely, one seeks $4\times4$ matrix factorizations of $F$ whose associated rank-two
lattices have reflexive hulls isomorphic to the modules $U_1,\ldots,U_\nu$. The following example illustrates both the direct
Koszul construction and the diagonal modification that is needed when the desired lattice does not arise directly from the monomial Koszul family.

\begin{ex}\label{D75-example}
Consider the dihedral quotient singularity $\mathbb D_{7,5}$. We have
$\frac75=[2,2,3]$, so there are two coefficient-$2$ vertices on the long branch. The
corresponding rank-two special CM modules are $U_1=V_3, U_2=V_1$. The remaining three vertices correspond to the rank-one special modules $W_+$, $W_-$ and $W_1$, which are obtained from the
$2\times2$ matrix factorizations constructed above. For this singularity,
$$
a_2=4,\qquad \alpha=1,\qquad s_3=1,\qquad \widehat s_3=0,\qquad \beta=1.$$
Hence $N=4$, $H=0$, $L=2$ and the projected hypersurface is $A=\mathbb C[X,Y,Z]/(F)$ with $F=Y^4+X^3Y^2-XZ^2$. We first describe the normalization explicitly. Put
$$a=uv,\qquad D=u^{10}-v^{10},\qquad S=u^{10}+v^{10}.$$
Up to multiplication of the generators by nonzero constants, the normalization map $A\hookrightarrow B$ is given by
$$
X=a^4,\qquad Y=\frac12aD,\qquad Z=\frac14SD.$$
Indeed, since $S^2-D^2=4a^{10}$ we have
\begin{align*}
Y^4+X^3Y^2-XZ^2&=\frac1{16}a^4D^4+\frac14a^{14}D^2-\frac1{16}a^4S^2D^2\\
&=\frac1{16}a^4D^2(D^2+4a^{10}-S^2)=0.
\end{align*}

\subsection{The module $U_1=V_3$} Consider $\theta_1=(2,1,1,0,0,1)$. The three monomials of $F$ are split as
$Y^4=(Y^2)(Y^2)$, $X^3Y^2=(XY)(X^2Y)$ and $-XZ^2=(-Z)(XZ)$. Thus
$$p_1=Y^2,\qquad q_1=Y^2, \qquad p_2=XY,\qquad q_2=X^2Y, \qquad p_3=-Z,\qquad q_3=XZ.$$
The associated three-term Koszul matrix factorization is
$$
\Phi_1=
\begin{pmatrix}
Y^2&-X^2Y&-XZ&0\\
XY&Y^2&0&-XZ\\
-Z&0&Y^2&X^2Y\\
0&-Z&-XY&Y^2
\end{pmatrix}, \qquad \Psi_1=
\begin{pmatrix}
Y^2&X^2Y&XZ&0\\
-XY&Y^2&0&XZ\\
Z&0&Y^2&-X^2Y\\
0&Z&XY&Y^2
\end{pmatrix}.
$$
Hence $\Phi_1\Psi_1=\Psi_1\Phi_1=F\textnormal{Id}_4$. Put $E_1=\operatorname{coker}(\Phi_1)$. By Proposition \ref{minimal-koszul-rank-two}, $E_1$ is a rank-two MCM $A$-module minimally generated by four elements. We now identify its reflexive normalization. Four covariant generators of $V_3$ may be chosen as
$$
h_1=\binom{u^3}{v^3},\qquad h_2=\binom{v^7}{-u^7},\qquad h_3=\binom{u^5v^2}{-u^2v^5},\qquad
h_4=\binom{u^2v^9}{u^9v^2}.$$
Writing these generators in a basis of the generic fibre and using the expressions for $X,Y,Z$ above, one obtains, after a change of
basis in $K^2$, multiplication by an element of $K^\times$, a permutation of the generators, and multiplication of generators by
nonzero constants, the fractional lattice
$$
V_3\cong \operatorname{im}\begin{pmatrix}
X^2Y&-XZ&Y^2&0\\
-Z&XY&0&Y^2
\end{pmatrix}.
$$
This is, up to a simultaneous permutation of the three pairs and multiplication of the factors by nonzero constants, the Koszul
lattice associated to the splitting above. Therefore Proposition \ref{lattice-general-koszul} gives
$(E_1\otimes_A B)^{\vee\vee}\cong V_3=U_1$.

\subsection{The module $U_2=V_1$} For the second rank-two special module a diagonal modification is
needed. Four covariant generators of $V_1$ may be chosen as
$$
g_1=\binom uv,\qquad g_2=\binom{u^3v^2}{-u^2v^3},\qquad g_3=\binom{v^9}{u^9},\qquad
g_4=\binom{u^2v^{11}}{-u^{11}v^2}.$$
Take $g_1,g_2$ as a basis of the generic fibre. Since $\textnormal{det}(g_1,g_2)=-2u^3v^3=-2a^3$ a direct calculation gives
$$
g_3=\frac{S}{2a}g_1-\frac{D}{2a^3}g_2,\qquad g_4=-\frac{aD}{2}g_1+\frac{S}{2a}g_2$$
Since
$$
\frac{S}{2a}=\frac ZY,\qquad \frac{D}{2a^3}=\frac YX,\qquad \frac{aD}{2}=Y$$
we obtain
$$
V_1\cong \operatorname{im}\begin{pmatrix}
1&0&\dfrac ZY&-Y\\
0&1&-\dfrac YX&\dfrac ZY
\end{pmatrix}$$
Multiplying this fractional lattice by $XY\in K^\times$, and then permuting and rescaling the generators, gives
$$
V_1\cong \operatorname{im} \begin{pmatrix}
-XZ&XY^2&-XY&0\\
-Y^2&XZ&0&-XY
\end{pmatrix}$$
This matrix has three-term Koszul form $\begin{pmatrix}
Q_1&Q_2&Q_3&0\\
-P_2&P_1&0&Q_3
\end{pmatrix}$ with
$$
P_1=XZ,\qquad Q_1=-XZ,\quad P_2=Y^2,\qquad Q_2=XY^2, \qquad P_3=-X^3Y,\qquad Q_3=-XY$$
The corresponding three products satisfy
$$P_1Q_1+P_2Q_2+P_3Q_3=-X^2Z^2+XY^4+X^4Y^2=X(Y^4+X^3Y^2-XZ^2)=XF$$
Thus these data determine a monomial Koszul matrix factorization $(\widetilde\Phi_2,\widetilde\Psi_2)$ of $XF$.
We now cancel the additional factor $X$. Put $D_X=\operatorname{diag}(1,X,1,X)$ and define
$$
\Phi_2=\widetilde\Phi_2D_X^{-1},\qquad \Psi_2=D_X\frac{\widetilde\Psi_2}{X}$$
The resulting matrices have polynomial entries and are given by
$$
\Phi_2=\begin{pmatrix}
XZ&-Y^2&XY&0\\
Y^2&-Z&0&Y\\
-X^3Y&0&-XZ&Y^2\\
0&-X^2Y&-Y^2&Z
\end{pmatrix},\qquad \Psi_2=\begin{pmatrix}
-Z&Y^2&-Y&0\\
-Y^2&XZ&0&-XY\\
X^2Y&0&Z&-Y^2\\
0&X^3Y&Y^2&-XZ
\end{pmatrix}$$
By construction, $\Phi_2\Psi_2=\Psi_2\Phi_2=FI_4$. Hence $(\Phi_2,\Psi_2)$ is a $4\times4$ matrix factorization of the
original projected equation $F$. Put $E_2=\operatorname{coker}(\Phi_2)$. To identify its reflexive normalization, consider the first two rows of $\Psi_{2,B}$:
$$T_2=\begin{pmatrix}
-Z&Y^2&-Y&0\\
-Y^2&XZ&0&-XY
\end{pmatrix}$$
and let $L_2=\operatorname{im}(T_2)\subseteq B^2$. 
Since $T_2\Phi_{2,B}=0$ there is a surjective homomorphism
$\operatorname{coker}(\Phi_{2,B})\longrightarrow L_2$. Both modules have rank two, so this map becomes an isomorphism after
tensoring with $K$. Its kernel is therefore torsion. Since $L_2$ is torsion-free, it follows that
$$(\operatorname{coker}(\Phi_{2,B}))^{\vee\vee}\cong L_2^{\vee\vee}$$
Hence $(E_2\otimes_A B)^{\vee\vee}\cong L_2^{\vee\vee}$. Finally, we have 
$$\begin{pmatrix}X&0\\
0&1
\end{pmatrix}
T_2=\begin{pmatrix}
-XZ&XY^2&-XY&0\\
-Y^2&XZ&0&-XY
\end{pmatrix}.
$$
Since $\begin{pmatrix}
X&0\\
0&1
\end{pmatrix} \in\operatorname{GL}_2(K)$ the fractional lattices $L_2$ and $V_1$ are isomorphic. As $V_1$ is
reflexive, $L_2^{\vee\vee}\cong V_1$. So we get 
$$(E_2\otimes_A B)^{\vee\vee}\cong V_1=U_2$$
Thus both rank-two special CM modules of $\mathbb D_{7,5}$ are recovered from $4\times4$ matrix factorizations of the projected
hypersurface:
$$(E_1\otimes_A B)^{\vee\vee}\cong U_1=V_3, \qquad (E_2\otimes_A B)^{\vee\vee}\cong U_2=V_1$$
The first arises directly from a monomial three-term Koszul factorization of $F$, while the second is obtained from a monomial
Koszul factorization of $XF$ by cancelling the additional factor $X$ through a diagonal modification.
\end{ex}

\noindent The example shows that rank-two special CM modules may arise in two different ways from the projected hypersurface. The module $U_1=V_3$ is obtained directly from a monomial three-term Koszul
matrix factorization of $F$, whereas $U_2=V_1$ is obtained from a Koszul factorization of the monomial multiple $XF$ by a diagonal
modification. In both cases, extension to the normalization followed by reflexive hull recovers the corresponding special CM module.

\begin{defn}\label{diagonal-cancellation}
Let $m$ be a monomial and let $(\widetilde\Phi,\widetilde\Psi)$ be a $4\times4$ matrix factorization
of $mF$, so that
$$
\widetilde\Phi\widetilde\Psi =\widetilde\Psi\widetilde\Phi =mF\textnormal{Id}_4$$
Let $D=\operatorname{diag}(d_1,d_2,d_3,d_4)$ be a diagonal matrix with monomial entries. Suppose that
$$
\Phi=\widetilde\Phi D^{-1}\qquad\text{and}\qquad \Psi=D\frac{\widetilde\Psi}{m}$$
have entries in $S$. We call $(\Phi,\Psi)$ the diagonal cancellation of $(\widetilde\Phi,\widetilde\Psi)$ with respect to $D$.
\end{defn}

\begin{prop}\label{diagonal-cancellation-mf} With the notation above, $(\Phi,\Psi)$ is a $4\times4$ matrix
factorization of $F$.
\end{prop}

\begin{proof} By construction,
$$
\Phi\Psi = \widetilde\Phi D^{-1}D\frac{\widetilde\Psi}{m} =\frac{1}{m}\widetilde\Phi\widetilde\Psi =F\textnormal{Id}_4$$
Similarly,
$$
\Psi\Phi =D\frac{\widetilde\Psi}{m}\widetilde\Phi D^{-1}=D(FI_4)D^{-1}=F\textnormal{Id}_4$$
Hence $(\Phi,\Psi)$ is a matrix factorization of $F$.
\end{proof}

\noindent Diagonal cancellation need not preserve the lattice associated to the original Koszul factorization: the diagonal entries
of $D$ need not be units along the height-one divisors of $B$. So, after cancellation the normalized cokernel must be
identified from the resulting factorization itself. The following observation provides the required comparison.

\begin{lem}\label{birational-lattice-double-dual}
Let $B$ be a normal domain with fraction field $K$. Let $M$ be a finitely generated $B$-module of rank $r$, and let
$L\subseteq K^r$ be a torsion-free $B$-module of rank $r$. Suppose that there is a surjective homomorphism
$\eta: M\longrightarrow L $ which becomes an isomorphism after tensoring with $K$. Then
$M^{\vee\vee}\cong L^{\vee\vee}$.
\end{lem}

\begin{proof} Since $\eta$ becomes an isomorphism after tensoring with $K$, its
kernel is torsion. Conversely, since $L$ is torsion-free, every torsion element of $M$ belongs to $ker(\eta)$. Hence
$ker(\eta)=\operatorname{tor}(M)$ and therefore $M/\operatorname{tor}(M)\cong L$. Since every homomorphism from a torsion $B$-module to the domain $B$ is zero,
$$ 
M^\vee \cong \bigl(M/\operatorname{tor}(M)\bigr)^\vee \cong L^\vee $$
Taking duals once more gives $M^{\vee\vee}\cong L^{\vee\vee}$. 
\end{proof}

\begin{prop}\label{general-rank-two-lattice-comparison} Let $(\Phi,\Psi)$ be a $4\times4$ matrix factorization of $F$, and
put $E=\operatorname{coker}(\Phi)$. Assume that $E$ has rank two. Let $T$ be a $2\times4$ matrix over
$B$ such that $T\Phi_B=0$ and assume that $T$ has generic rank two. If
$L=\operatorname{im}(T)\subseteq B^2$ then $(E\otimes_A B)^{\vee\vee}\cong L^{\vee\vee}$. In particular, if $L^{\vee\vee}\cong U$ for a rank-two special CM $B$-module $U$, then
$(E\otimes_A B)^{\vee\vee}\cong U$.

\end{prop}

\begin{proof}
The relation $T\Phi_B=0$ induces a surjective homomorphism
$\eta: \operatorname{coker}(\Phi_B)\longrightarrow L$. Since $E\otimes_A B\cong\operatorname{coker}(\Phi_B)$
the source has rank two. By assumption, $T$ has generic rank two, so $L$ also has rank two. Hence $\eta$ becomes an isomorphism after tensoring with $K$. Lemma \ref{birational-lattice-double-dual} gives
$$(\operatorname{coker}(\Phi_B))^{\vee\vee}\cong L^{\vee\vee}$$
and we get 
$$(E\otimes_A B)^{\vee\vee} \cong L^{\vee\vee}.$$
The final assertion is easy.
\end{proof}

\noindent Combining the Koszul construction with diagonal cancellation gives the following recovery criterion.

\begin{thm}\label{rank-two-diagonal-recovery}
Let $U$ be an indecomposable rank-two special CM $B$-module. Assume that, after a change of basis in $K^2$ and multiplication by an element of $K^\times$, $U$ admits a four-generator fractional lattice presentation
$$
U\cong \operatorname{im}\begin{pmatrix}
Q_1&Q_2&Q_3&0\\
-P_2&P_1&0&Q_3
\end{pmatrix}
\subseteq K^2$$ where the $P_i,Q_i$ are monomials satisfying
$P_1Q_1+P_2Q_2+P_3Q_3=mF$ for some monomial $m$. Let $(\widetilde\Phi,\widetilde\Psi)$ be the corresponding three-term
Koszul matrix factorization of $mF$. Suppose that there exists a diagonal monomial matrix $D$ such that
$\Phi=\widetilde\Phi D^{-1},\qquad \Psi=D\frac{\widetilde\Psi}{m}$ have entries in $S$, and assume that
$E:=\operatorname{coker}(\Phi)$ has rank two. Assume moreover that a $2\times4$ submatrix $T$ of
$\Psi_B$ has generic rank two and satisfies $C\,\operatorname{im}(T)\cong U$ for some $C\in\operatorname{GL}_2(K)$. Then
$(E\otimes_A B)^{\vee\vee}\cong U$. 
\end{thm}

\begin{proof}
By Proposition \ref{diagonal-cancellation-mf}, $(\Phi,\Psi)$ is a matrix factorization of $F$. Since $T$ consists of
rows of $\Psi_B$, the relation $\Psi_B\Phi_B=0$ gives $T\Phi_B=0$.
By assumption $T$ has generic rank two. Proposition \ref{general-rank-two-lattice-comparison} gives
$$(\operatorname{coker}(\Phi)\otimes_A B)^{\vee\vee}\cong (\operatorname{im}(T))^{\vee\vee}$$
Since $C\in\operatorname{GL}_2(K)$ identifies $\operatorname{im}(T)$ with $U$ as a fractional $B$-lattice and $U$
is reflexive,  $(\operatorname{im}(T))^{\vee\vee}\cong U$. Hence we have
$$(\operatorname{coker}(\Phi)\otimes_A B)^{\vee\vee}\cong U.$$
\end{proof}

\begin{cor}\label{D75-all-specials} For $\mathbb D_{7,5}$ every indecomposable special CM module is
obtained from a matrix factorization of the projected hypersurface
$F=Y^4+X^3Y^2-XZ^2$ by extension to the normalization followed by reflexive hull. More precisely, the rank-one special modules
$W_+, W_-, W_1$ are obtained from the $2\times2$ matrix factorizations constructed
above, while the two rank-two special modules satisfy
$$(E_1\otimes_A B)^{\vee\vee}\cong U_1=V_3, \qquad (E_2\otimes_A B)^{\vee\vee}\cong U_2=V_1$$
Here $E_1$ is obtained directly from a monomial three-term Koszul factorization of $F$, whereas $E_2$ is obtained from a monomial
Koszul factorization of $XF$ by the diagonal cancellation $D_X=\operatorname{diag}(1,X,1,X)$.
\end{cor}

\noindent The example exhibits the two mechanisms that occur in the rank-two construction. A special module may arise directly
from the reflexive normalization of a monomial Koszul factorization of the projected equation, or it may first appear as the lattice of a
Koszul factorization of a monomial multiple of $F$ and then be recovered from a matrix factorization of $F$ by diagonal
cancellation. In either case, Proposition \ref{general-rank-two-lattice-comparison} reduces the identification
of the normalized cokernel to the corresponding rank-two fractional lattice. Together with the $2\times2$ constructions for the rank-one special modules, this completes the matrix-factorization description for $\mathbb D_{7,5}$ and provides the general recovery criterion used for rank-two modules.

\noindent The example shows that rank-two special CM modules may arise in two different ways from the projected hypersurface. The module $U_1=V_3$ is obtained directly from a monomial three-term Koszul matrix factorization of $F$, whereas $U_2=V_1$ is obtained from a Koszul factorization of the monomial multiple $XF$ by a diagonal
modification. In both cases, extension to the normalization followed by reflexive hull recovers the corresponding special CM module. 

This identification can be imposed for the families of dihedral quotient singularities as we discuss below.

\subsection{Classical $\mathbb D_n$ as a special case}

\noindent The rational double point of type $\mathbb D_n$, $n\geq4$, is the hypersurface
$R={\mathbb C[x,y,z]}/{(x^{n-1}+xy^2+z^2)}$. Since $R$ is a hypersurface, its MCM modules are described by matrix factorizations of the defining equation. The indecomposable matrix
factorizations of type $\mathbb D_n$ are classical and are explicitly described in \cite{KajiuraSaito} ( see also \cite{Yoshino}). 

\noindent Up to isomorphism, there are $n$ nonfree indecomposable MCM
$R$-modules $M_1,\ldots, M_n$. The modules $M_1,M_{n-1},M_n$ have rank one and are represented by
minimal $2\times2$ matrix factorizations and $M_2,\ldots, M_{n-2}$ have rank two and are represented by minimal
$4\times4$ matrix factorizations. Thus there are exactly $n-3$ indecomposable rank-two MCM modules, corresponding to the vertices which occur with coefficient $2$ in the fundamental cycle.

\noindent This is precisely the rank pattern appearing in the general dihedral case: vertices with coefficient $1$ in the fundamental cycle correspond to rank-one special CM modules, while vertices with
coefficient $2$ correspond to rank-two special CM modules.

\end{document}